\documentclass[12pt, reqno]{amsart}
\usepackage[margin=1in]{geometry}

\usepackage{amsfonts}
\usepackage{amsmath}
\usepackage{amsthm}
\usepackage{amssymb}
\usepackage{url, verbatim}
\usepackage{graphicx}
\usepackage[all]{xy}
\usepackage{enumitem}

\newtheorem{theorem}{Theorem}[section]
\newtheorem{lemma}[theorem]{Lemma}
\newtheorem{proposition}[theorem]{Proposition}
\newtheorem{corollary}[theorem]{Corollary}

\newtheorem{predefinition}[theorem]{Definition}
\newenvironment{definition}{\begin{predefinition}\rm}{\end{predefinition}}
\newtheorem{preremark}[theorem]{Remark}
\newenvironment{remark}{\begin{preremark}\rm}{\end{preremark}}
\newtheorem{prenotation}[theorem]{Notation}
\newenvironment{notation}{\begin{prenotation}\rm}{\end{prenotation}}
\newtheorem{preexample}[theorem]{Example}
\newenvironment{example}{\begin{preexample}\rm}{\end{preexample}}
\newtheorem{preclaim}[theorem]{Claim}

\newtheorem{prequestion}[theorem]{Question}

\makeatletter

\DeclareMathOperator*{\op}{\oplus}

\newcommand{\QQ}{\mathbb{Q}}

\newcommand{\C}{\mathbf{C}}
\newcommand{\NN}{\mathbb{N}}

\newcommand{\ZZ}{\mathbb{Z}}

\newcommand{\FF}{\mathbb{F}}

\newcommand{\LL}{\mathcal{L}}

\newcommand{\Proj}{\mathbb{P}}

\DeclareMathOperator{\End}{End}

\def\del{{\partial}}

\newcommand \Pic {{\mathop{\rm Pic}}}

\newcommand{\sing}{{\rm Sing}}

\newcommand{\aut}{{\rm Aut}}

\newcommand{\smooth}{{\rm sm}}

\newcommand{\calm}{{\mathcal M}}

\newcommand{\calh}{{\mathcal H}}

\newcommand{\til}[1]{{\widetilde{#1}}}

\renewcommand{\bar}[1]{{\overline{#1}}}

\newcommand{\calt}{{\mathcal T}}

\usepackage[usenames,dvipsnames]{color}

\title[The $p$-rank $0$ stratum of the moduli space of cyclic covers]
{The boundary of the $p$-rank $0$ stratum of the moduli space of cyclic covers of the projective line}
\date{}

\author{Ekin Ozman}
\address{Bogazici University, Faculty of Arts and Sciences, Bebek, Istanbul, 34342, Turkey}
\email{ekin.ozman@boun.edu.tr}

\author{Rachel Pries}
\address{Department of Mathematics, 
Colorado State University, 
Fort Collins, CO 80523-1874, USA}
\email{pries@colostate.edu}

\author{Colin Weir}
\address{The Tutte Institute for Mathematics and Computing, Ottawa, Ontario, Canada}
\email{colinoftheweirs@gmail.com}

\thanks{Pries was supported by 
NSF grants DMS-11-01712 and DMS-15-02227 during the initial development of this project and 
by NSF grant DMS-19-01819 during its completion.  Ozman was partially supported by 
Scientific and Research Council of Turkey [TUBITAK-117F274].
We thank the anonymous referees for helpful comments.}

\begin{document}

\maketitle

\begin{abstract}
We study the $p$-rank stratification of the moduli space of cyclic degree $\ell$ covers of the projective line 
in characteristic $p$ for distinct primes $p$ and $\ell$.  
The main result is about the intersection of the $p$-rank $0$ stratum with the boundary of the moduli 
space of curves.  When $\ell=3$ and $p \equiv 2 \bmod 3$ is an odd prime,
we prove that there exists a smooth trielliptic curve in characteristic $p$, 
for every genus $g$, signature type $(r,s)$ and $p$-rank $f$ satisfying the clear necessary conditions.

\end{abstract}
	
\section{Introduction} \label{sec:intro}

Suppose $Y$ is a smooth projective connected curve of genus $g$ defined over an algebraically closed field $k$ of 
characteristic $p >0$. 
The {\em $p$-rank} of $Y$ is the integer $f$ such that $p^f$ is the number of $p$-torsion points of the Jacobian of $Y$.  It is known that $0 \leq f \leq g$.

Let $g \geq 2$.
Consider the moduli space $\calm_g$ of smooth curves of genus $g$ over $k$ and its Deligne-Mumford compactification
$\overline{\calm}_g$.
Consider the boundary $\delta \calm_g = \overline{\calm}_g - \calm_g$ 
of $\calm_g$; its points represent singular stable curves of genus $g$.

It is a compelling problem to understand the geometry of 
the $p$-rank $f$ stratum $\overline{\calm}_g^f$ of $\overline{\calm}_g$. 
For example, in most cases, it is not known whether $\overline{\calm}_g^f$ is irreducible.

It is known, by \cite[Theorem 2.3]{FV}, that every irreducible component of $\overline{\calm}_g^f$ has dimension $2g- 3 + f$.  
By \cite[Lemma 3.2]{AP:monoprank}, every irreducible component of $\overline{\calm}_g^f$ contains an open dense subset which lies in $\calm_g^f$.
It follows that there exists a smooth curve of genus $g$ and $p$-rank $f$ defined over $\overline{\mathbb F}_p$ 
for every prime $p$ and pair of integers $g$ and $f$ such that $0 \leq f \leq g$.

The proof of \cite[Theorem 2.3]{FV} uses properties of the intersection of the $p$-rank strata with the boundary.
By \cite[Lemma 2.5]{FV}, see also \cite[Corollary 3.6]{AP:monoprank}, 
every irreducible component ${\mathcal S}$ of the $p$-rank $0$ stratum $\overline{\calm}_g^0$ 
intersects $\delta \calm_g$; specifically:  \\
(i) ${\mathcal S}$ contains points that represent \emph{chains} of $g$
(supersingular) elliptic curves; and\\
(ii) ${\mathcal S}$ intersects every irreducible component of $\delta \calm_g$.

Analogously, for odd $p$, one can study the $p$-rank $f$ stratum $\overline{\calh}_g^f$ of the moduli space $\calh_g$ of hyperelliptic curves of genus $g$.
By \cite[Proposition 2]{GP}, every irreducible component of $\overline{\calh}_g^f$ has dimension $g-1+f$. 
Every irreducible component of $\overline{\calh}_g^f$ contains an open dense subset which lies in $\calh_g^f$ \cite[Lemma~3.2]{APhypmono}.
It follows that there exists a smooth hyperelliptic curve of genus $g$ 
and $p$-rank $f$ defined over $\overline{\mathbb F}_p$ 
for every odd prime $p$ and pair of integers $g$ and $f$ such that $0 \leq f \leq g$.

The proof of these facts for the hyperelliptic locus also uses the intersection of the $p$-rank strata 
with the boundary $\delta \calh_g = \overline{\calh}_g - \calh_g$. 
Every irreducible component ${\mathcal S}$ of the $p$-rank $0$ stratum $\overline{\calh}_g^0$ intersects 
$\delta \calh_g$ by \cite[Theorem 3.11]{APhypmono}; specifically: \\
(i)' ${\mathcal S}$ contains points that represent \emph{trees} of $g$
(supersingular) elliptic curves. 

However, it is not known whether ${\mathcal S}$ intersects every irreducible component of
$\delta \calh_g$.

In this paper, for an odd prime $\ell$ with $\ell \not = p$, 
we study analogous questions about the $p$-ranks
of curves that admit a $\ZZ/\ell \ZZ$-cover 
of the projective line ${\mathbb P}^1$. 
Let $\calt_{\ell, g}$ denote the moduli space of such $\ZZ/\ell \ZZ$-covers
and let $\bar{\calt}_{\ell, g}$ denote its compactification.
The irreducible components of 
$\calt_{\ell, g}$ are indexed not only by the degree $\ell$ and the genus $g$, but 
also by the discrete data of the inertia type, which determines the signature type.
In Proposition~\ref{Tpurity}, 
we compute a lower bound for the dimension of each irreducible component of the $p$-rank strata of 
$\calt_{\ell, g}$, in terms of the signature type.  

The first topic we study is how the $p$-rank $0$ stratum $\bar{\calt}_{\ell, g}^0$ of the moduli space $\bar{\calt}_{\ell, g}$ intersects the boundary $\delta \calm_g$.
The first main result of the paper is:
\begin{theorem} \label{Tintro} [See Theorem~\ref{Cbreakupf=0}]
Every irreducible component ${\mathcal S}$ of $\bar{\calt}_{\ell, g}^0$ contains a point representing a curve 
of compact type which has at least ${\rm dim}({\mathcal S}) + 1$ components.
\end{theorem}

The geometric conclusion from Theorem~\ref{Tintro} is not as strong 
as the analogous result in the hyperelliptic case.
This prevented us from using Theorem~\ref{Tintro} to find the dimension of the irreducible components of $\calt_{\ell, g}^f$ in general, see Remark~\ref{Rsignificance} and Section~\ref{Sboundaryproblem}.

For this reason, in Section~\ref{Strielliptic}, we specialize to the case $\ell = 3$.
In Proposition~\ref{Cf=B-2}, for all $g \geq 2$ and all primes $p \geq 5$, 
we generalize a result of Bouw by proving
that every component of the moduli space of trielliptic curves of genus $g$ contains a 
curve whose $p$-rank is not the maximum.
Then we prove:
\begin{theorem} \label{Tintro2} [See Theorem~\ref{Trsexistence}]
For every odd $p \equiv 2 \bmod 3$, every $g \in \NN$, every trielliptic signature type $(r,s)$ for $g$, 
and every $f$ (satisfying the clear necessary conditions that $f$ is even and $0 \leq f \leq 2{\rm min}(r,s)$), 
there exists a smooth trielliptic curve defined over $\bar{\FF}_p$ 
with genus $g$, signature type $(r,s)$ and $p$-rank $f$; furthermore,  
the dimension of at least one irreducible component of the $p$-rank 
$f$ stratum of $\calt_{3, g, (r,s)}$
equals the lower bound from Proposition~\ref{Tpurity}. 
\end{theorem}

In Corollary~\ref{Cr2total}, we strengthen Theorem~\ref{Tintro2} 
when $g$ is small for all odd primes $p \equiv 2 \bmod 3$
using an application of Theorem~\ref{Tintro}.

\section{Background} \label{sec:bckgrnd}

In this section, we include necessary material about cyclic covers and the $p$-rank.

\subsection{Stable $\ZZ/\ell \ZZ$-covers of a genus zero curve} 

Let $k$ be an algebraically closed field of characteristic $p > 0$ and 
let $S$ be an irreducible scheme over $k$.
Let $G=\ZZ/\ell \ZZ$ be a cyclic group of odd prime order $\ell \not = p$.  Let $G^*=G-\{0\}$.

Let $\psi:Y \to S$ be a semi-stable curve.  
If $s \in S$, let $Y_s$ denote the fiber of $Y$ over $s$.  Let $\sing_S(Y)$
be the set of $z \in Y$ for which $z$ is a singular point of the fiber
$Y_{\psi(z)}$.

A {\em mark} $R_\Xi$ on $Y/S$ is a closed subscheme of
$Y-\sing_S(Y)$ which is finite and \'etale over $S$.    The {\em
  degree} of $R_\Xi$ is the number of points in any geometric fiber of $R_\Xi \to S$.
A marked semi-stable curve $(Y/S, R_\Xi)$ is {\em stably marked} if
every geometric fiber of $Y$ satisfies the following condition: for
each irreducible component $Y_0$ of genus zero, 
$\# (Y_0 \cap (\sing_S(Y) \cup R_\Xi)) \ge 3$;
(if the fibers of $Y/S$ have genus $1$, we also assume that the degree of $R_\Xi$ is positive).

Consider a $G$-action $\iota_0: G \hookrightarrow \aut_S(Y)$ on $Y$.  
Let $R$ denote the ramification locus of $Y \to Y/\iota_0(G)$. 
The smooth ramification locus is $R_{\smooth} := R - (R\cap \sing_S(Y))$.  We say
that $(Y/S,\iota_0)$ is a {\em stable $G$-curve} if $Y/S$ is a
semi-stable curve; if $\iota_0: G \hookrightarrow \aut_S(Y)$ is an action of $G$;
if $R_{\smooth}$ is a mark on $Y/S$; and
if $(Y/S,R_{\smooth})$ is stably marked.  

If $z \in \sing_S(Y)$, let  
$Y_{z,1}$ and $Y_{z,2}$ denote the two components of the formal completion of $Y_{\psi(z)}$ at $z$.
A stable $G$-curve $(Y/S, \iota_0)$ is {\em admissible} if
the following conditions are satisfied for every geometric point $z \in R\cap{\rm Sing}_S(Y)$:  
\begin{enumerate}
\item $\iota_0(1)$ stabilizes each branch $Y_{z,i}$; 
\item $z$ is a ramification point of the restriction of $\iota_0$ to $Y_{z,i}$;
\item the characters of the action of $\iota_0$ on the tangent spaces
of $Y_{z,1}$ and $Y_{z,2}$ at $z$ are inverses. 
\end{enumerate}

Suppose that $(Y/S, \iota_0)$ is an admissible stable $G$-curve.
Then $Y/\iota_0(G)$ is also a stably marked curve \cite[Proposition~1.4]{ekedahlhurwitz}.  
The mark on $Y/\iota_0(G)$ is the smooth branch locus $B_{\smooth}$, 
which is the (reduced subscheme of) the image of $R_{\smooth}$ under the morphism $Y \to Y/\iota_0(G)$.
Let $n$ be the degree of $R_{\smooth}$ (the number of smooth ramification points).  
We suppose from now on that $Y/\iota_0(G)$ has arithmetic genus $0$.
By the Riemann-Hurwitz formula, the arithmetic
genus of each fiber of $Y$ is 
\begin{equation} \label{Egenus}
g=(n-2)(\ell-1)/2.
\end{equation}

\subsection{The inertia type and signature type}

Let $s$ be a geometric point of $S$
and let $a$ be a point of the fiber $R_{\smooth,s}$.  Then $G= \ZZ/\ell\ZZ$ acts on the tangent space of $Y_s$ at
$a$ via a character $\chi_a: G \to k^*$.  In particular, there is
a unique choice of $\gamma_a \in (\ZZ/\ell \ZZ)^*$ so that $\chi_a(1)
= \zeta_\ell^{\gamma_a}$.  
We say that $\gamma_a$ is the {\em canonical generator of inertia} at $a$.  
The {\em inertia type} of $(Y/S,\iota_0)$ is the multiset
$\bar \gamma = \{\gamma_{a}  \mid  a \in R_{\smooth,s}\}$.   It is independent of the choice of $s$.
By Riemann's existence theorem, there exists a cover $(Y, \iota_0)$ with inertia type $\{\gamma_{a} \mid  a \in R_{\smooth,s}\}$ 
if and only if 
$\sum_{a \in R_{\smooth,s}} \gamma_a = 0 \in \ZZ/\ell \ZZ$ \cite[Theorem 2.13]{Vbook}.

A labeling of a mark $R_\Xi$ of degree $n$ is a bijection $\eta$ between $\{1,
  \cdots, n\}$ and the irreducible components of $R_\Xi$.
A labeling of an admissible stable $G$-curve $(Y/S,\iota_0)$ is a
labeling $\eta$ of $R_{\smooth}$.
There is an induced labeling $\eta_0:\{1, \ldots, n\} \to B_{\smooth}$.

If $(Y/S,\iota_0,\eta)$ is a labeled $G$-curve, the {\em class vector}
is the map of sets $\gamma:\{1, \ldots, n\} \to G^*$ such that
$\gamma(i) = \gamma_{\eta(i)}$.  We write $\gamma = (\gamma(1),
\ldots, \gamma(n))$.  If $\gamma$ is a class vector, we denote
its inertia type by $\bar\gamma: G^* \to \ZZ_{\ge 0}$ 
where $\bar\gamma(h) = \#\{i \mid 1 \leq i \leq n, \ \gamma(i) = h\}$ for all $h\in G^*$.

Let $\zeta_{\ell} \in k$ be a primitive $\ell$th root of unity.  
The automorphism $\iota_0(1)$ induces an action on $H^0(Y_s, \Omega^1)$.
Let $\LL_i$ be the $\zeta_\ell^i$-eigenspace of $H^0(Y_s, \Omega^1)$, for $0 \leq i \leq \ell-1$.
There is an eigenspace decomposition:
\[H^0(Y_s, \Omega^1)=\op\limits_{i=0}^{\ell-1} \LL_i.\] 
Let $s_i ={\rm dim}(\LL_i)$.  Then $\LL_0=\{0\}$ and $s_0=0$ 
since $Y/\iota_0(G)$ has genus $0$.
The {\em signature type} is $(s_1, \ldots, s_{\ell-1})$.
It is locally constant on $S$.
For an integer $t$, let $\langle \frac{t}{\ell} \rangle = \frac{t}{\ell} - \lfloor \frac{t}{\ell} \rfloor$ denote the 
fractional part of $\frac{t}{\ell}$. 

\begin{lemma} \label{Lsignature}
For $1 \leq i \leq \ell -1$, 
\begin{equation*} 
s_i=-1 + \sum\limits_{j=1}^{n} \langle \frac{- i a_j}{\ell} \rangle. 
\end{equation*}
\end{lemma}

\begin{proof}
This can be found in \cite[Lemma 2.7, Section 3.2]{moonen}, or deduced 
from earlier results \cite[Lemma 4.3]{Bouw} (the negative sign does not appear for the action on ${\mathcal O}$) 
or \cite[Proposition~1]{kani}.
\end{proof}

\subsection{Restrictions on the $p$-rank}

Let $\mu_p$ be the kernel of Frobenius on ${\mathbb G}_m$.
The {\em $p$-rank} of a semi-abelian variety $A'$ over $k$ is $f_{A'}={\rm dim}_{{\mathbb F}_p}{\rm Hom}(\mu_p, A')$. 
If $A'$ is an extension of an abelian variety $A$ by a torus $T$, 
then $f_{A'}=f_A + {\rm rank}(T)$.

For an abelian variety $A$,     
the $p$-rank can also be defined as the integer $f_A$ such that 
the number of $p$-torsion points in $A(k)$ is $p^{f_A}$.
If $A$ has dimension $g_A$ then $0 \leq f_A \leq g_A$.
The $p$-rank of a stable curve $Y$ is that of ${\rm Pic}^0(Y)$.

Recall that $\ell \not = p$ is prime.  Let $e$ be the order of $p$ in the multiplicative group  $(\mathbb Z/ \ell \mathbb Z)^\times$. 

\begin{lemma}\label{lem:edivf} 
Suppose that $Y \to \Proj_k^1$ is a $\ZZ/\ell \ZZ$-cover.  Let $f$ be the $p$-rank of $Y$.
\begin{enumerate}
\item Then $f$ is divisible by $e$, the order of $p$ modulo $\ell$.
\item If $\ell > 3$ or if $\ell =3$ and $p \equiv 1 \bmod 3$, then $f \not = g-1$.
\end{enumerate}
\end{lemma}

\begin{proof} 
The action of $\ZZ/\ell \ZZ$ on $Y$ induces an action of $\zeta_\ell$ on $J={\rm Jac}(Y)$ and its $p$-divisible group $J[p^\infty]$.
So $\ZZ[\zeta_\ell] \hookrightarrow {\rm End}(J[p^\infty])$.
If $(c, d)$ is a pair of relatively prime non-negative integers, and $\lambda=d/(c+d)$, 
let $G_{\lambda}$ denote a $p$-divisible group of codimension $c$, dimension $d$, and thus height $c + d$. 
By \cite{maninthesis}, the Dieudonn\'e-Manin classification, there is an isogeny of $p$-divisible groups
$J[p^\infty] \sim \oplus_\lambda G_\lambda^{m(\lambda)}$.
The action of $\zeta_\ell$ stabilizes every slope factor $G_{\lambda}^{m(\lambda)}$  of $J[p^\infty]$.
Hence $\ZZ[\zeta_\ell] \hookrightarrow \End(G_\lambda^{m(\lambda)})$.
If $m(\lambda)>0$, 
this yields an inclusion $\QQ(\zeta_\ell) \otimes \QQ_p \hookrightarrow {\rm Mat}_{m(\lambda)} (D_{\lambda}) $ where $D_{\lambda}$ 
is the $\QQ_p$-division algebra with Brauer invariant $\lambda \in \QQ/\ZZ$. 

\begin{enumerate}
\item If $\lambda=0$ then $m(\lambda)=f$.
Let $L_p$ be the completion of $L=\QQ(\zeta_\ell)$ at a prime lying above $p$.
Then $[L_p: \QQ_p]=e$ divides $m(\lambda)$. 

\item 
Suppose $f=g-1$.  Then the slope $1/2$ factor of $J[p^\infty]$ is the $p$-divisible group of a supersingular elliptic curve.
So $\QQ(\zeta_\ell) \subset E$ where $E$ is the endomorphism algebra of a supersingular elliptic curve.
Then $E$ is a quaternion algebra ramified exactly at $\infty, p$.
The only number fields contained in $E$ are quadratic fields inert or ramified at $p$. 
This gives a contradiction if $\ell >3$ or if $\ell =3$ and $p \equiv 1 \bmod 3$. \qedhere
\end{enumerate}
\end{proof}

The $p$-rank of $Y$ equals the stable rank of the Cartier operator $\C$.
If $\omega \in H^0(Y, \Omega^1)$, then $\C(\zeta_\ell^{pi} \omega)=\zeta_\ell^i\C(\omega)$.
Then $\C(\LL_i) \subset \LL_{\sigma(i)}$ where $\sigma$ is the permutation of $\ZZ/\ell\ZZ-\{0\}$ which sends $i$ to $p^{-1}i$. 
The cycle structure of $\sigma$ is determined by the splitting of $p$ in $\ZZ[\zeta_{\ell}]$.
Recall that $e$ is the order of $p$ modulo $\ell$. 
There are $(\ell-1)/e$ primes of $\ZZ[\zeta_\ell]$ lying over $p$ with residual degree $e$. 
Each orbit of $\C$ on $\{\LL_i\}$ has cardinality $e$.  Let $O$ denote the set of orbits.  
The contribution to the $p$-rank from each of the $e$ eigenspaces in an orbit $o \in O$ 
is bounded by the minimum of $s_i ={\rm dim}(\LL_i)$ for $\LL_i$ in $o$. 

Bouw used these ideas to find an upper bound on the $p$-rank,
which depends only on $p$, $\ell$, and the inertia type $\bar \gamma$; it is:
\begin{equation} \label{Ebound}
B(\bar \gamma):= \sum_{o \in O} e \cdot {\rm min}\{s_i \mid \LL_i \in o\}.
\end{equation}

\begin{theorem} \label{LboundB} (Bouw)
The integer $B(\bar \gamma)$ is an upper bound for the $p$-rank of a $\ZZ/\ell \ZZ$-cover of ${\mathbb P}^1$
with inertia type $\bar \gamma$, \cite[page 300, (1)]{Bouw}.
This upper bound $B(\bar \gamma)$ 
occurs as the $p$-rank of a $\ZZ/\ell \ZZ$-cover of ${\mathbb P}^1$ with inertia type $\bar \gamma$ 
if $p \geq \ell(n-3)$ \cite[Theorem 6.1]{Bouw}, 
or if $p = \pm 1 \bmod \ell$ \cite[Propositions~7.4, 7.8]{Bouw}, or if $n=4$ \cite[Proposition~7.7]{Bouw}.
\end{theorem}

\section{The moduli space of $\ZZ/\ell \ZZ$-covers and its boundary}

In this section, we introduce the material needed to study $p$-ranks of cyclic covers from a moduli-theoretic approach.
Recall that $G=\ZZ/\ell \ZZ$.

\subsection{Moduli spaces of stable $\ZZ/\ell \ZZ$-covers} \label{SAmoduli}

We define moduli functors on the category of schemes over $k$ whose $S$-points represent 
the listed objects:
\begin{enumerate}
\item $\overline{\calt}_{\ell, g}$: admissible 
stable $G$-curves $(Y/S, \iota_0)$ with $Y/S$ of genus $g$.
\item $\til{\calt}_{\ell, g}$: $(Y/S, \iota_0)$ as above, together with a labeling $\eta$ of the smooth 
ramification locus.
\item $\til{\calt}_{\ell, g; t}$:  $(Y/S, \iota, \eta)$ as above, together
with a mark $R_\Xi$ of degree $t$ such that $(Y/S,R_\Xi)$ is stably marked.
\end{enumerate}

Let $\calt_{\ell,g} \subset \overline{\calt}_{\ell, g}$ be the sublocus representing smooth $G$-curves.

Let $\gamma: \{1, \ldots, n\} \to (\ZZ/\ell \ZZ)^*$ be a class vector
of length $n=n(\gamma)$.  By \eqref{Egenus}, $\gamma$ determines the genus $g=g(\gamma)=(n-2)(\ell-1)/2$.
Let $\til{\calt}_{\ell, \gamma} \subset \til\calt_{\ell,g}$ be the substack
for which $(Y/S,\iota_0, \eta)$ has class vector $\gamma$. 
Let $\bar\calt_{\ell, \bar\gamma} \subset \bar\calt_{\ell,g}$ be the substack for
which $(Y/S,\iota_0)$ has inertia type $\bar\gamma$.

If two class vectors $\gamma$ and $\gamma'$ yield the same inertia type, so that
$\bar{\gamma} = \bar{\gamma}'$, then there is a permutation $\varpi$ of
$\{1, \ldots, n\}$ such that $\gamma' = \gamma \circ \varpi$.  This
relabeling of the branch locus yields an isomorphism
$\til\calt_{\ell, \gamma}  \simeq  \til \calt_{\ell, {\gamma \circ \varpi}}$.
Suppose $\gamma$ and $\gamma'$ differ by an
automorphism of $G$, so that there exists $\tau \in \aut(G)$ such
that $\gamma'= \tau \circ \gamma$.  This relabeling of
the $G$-action yields an isomorphism 
$\til\calt_{\ell, \gamma} \simeq  \til \calt_{\ell, \tau \circ \gamma}$.

\begin{lemma} \label{Lmoduli}
\begin{enumerate}
\item $\til\calt_{\ell,g}$ and $\bar\calt_{\ell, g}$ are smooth, proper
Deligne-Mumford stacks over $k$.  
\item $\calt_{\ell, g}$ is open and dense in $\bar \calt_{\ell, g}$.
\item The forgetful functor $\til\calt_{\ell, \gamma} \to
\bar\calt_{\ell, \bar\gamma}$ is \'etale and Galois.
\item $\bar \calt_{\ell, \bar \gamma}$ is an irreducible component of $\bar \calt_{\ell, g}$.
\item The dimension of $\calt_{\ell,{\bar \gamma} }$ is $n-3$.
\end{enumerate}
\end{lemma}

\begin{proof}
See \cite[Lemmas 2.2, 2.3, 2.4]{AP:tri}.
\end{proof}

\subsection{Clutching maps}

We review the clutching maps $\kappa_{g_1,g_2}$ and 
$\lambda_{g_1,g_2}$ of \cite{knudsen2}. Each of these is the restriction of a finite, unramified morphism
between moduli spaces of labeled curves.
They can be described in terms of their images on $S$-points for
an arbitrary $k$-scheme $S$.  We give explicit descriptions only
for sufficiently general $S$-points and defer to \cite{knudsen2} for
complete definitions.
A stable curve $Y$ has {\em compact type} 
if its dual graph is a tree or, equivalently, if ${\rm Pic}^0(Y)$ is
represented by an abelian scheme. 

For $i=1,2$, let $\gamma_i$ denote a class vector with length $n_i=n_i(\gamma_i)$
and let $g_i=g(\gamma_i)$.

\subsubsection{Clutching maps (compact type)}

There is a closed immersion \cite[Corollary 3.9]{knudsen2}
\[
\kappa_{g_1,g_2}: \bar \calm_{g_1; t_1} \times \bar \calm_{g_2; t_2}  \to \bar \calm_{g_1+g_2; t_1+t_2-2}.
\]
This clutching map
glues two curves $Y_1/S$ and $Y_2/S$ together to form a curve $Y/S$ by
identifying the last section of $Y_1$ and the first section of $Y_2$ in an ordinary double point.

As seen in \cite[Section 2.3]{AP:tri}, the clutching map extends to the moduli space of labeled $\ZZ/\ell \ZZ$-curves
as follows.  Let $g=g_1+g_2$ and $n=n_1+n_2-2$ and
$$\gamma=(\gamma_1(1), \ldots, \gamma_1(n_1-1), \gamma_2(2), \ldots, \gamma_2(n_2)).$$  
If $(Y_i/S, \iota_{0,i}, \eta_i)$ is a labeled $G$-curve with class vector $\gamma_i$, for $i=1,2$, 
then the clutched curve $Y/S$ has genus $g$ and admits a $G$-action $\iota_0$ and a labeling $\eta$
with class vector $\gamma$.
Moreover, $Y/S$ can be deformed to a smooth $G$-curve if and only if the $G$-action is admissible,
i.e., if and only if $\gamma_1(n_1)$ and $\gamma_2(1)$ are inverses \cite[Proposition 2.2]{ekedahlhurwitz}.  
In this situation, we write 
\begin{equation*}
\kappa_{g_1,g_2}: \til\calt_{\ell, \gamma_1} \times \til\calt_{\ell, \gamma_2}  \to \til\calt_{\ell, \gamma}.
\end{equation*}

By \cite[Ex.\ 9.2.8]{blr},
$\Pic^0(Y) \simeq \Pic^0(Y_1) \times \Pic^0(Y_2)$.
In particular, the $p$-rank of $Y$ is:
\begin{equation}
\label{eqblrprank}
f(Y) = f(Y_1)+f(Y_2).
\end{equation}

The signature type of $(Y/S, \iota_0)$ is the sum of those for $(Y_i/S, \iota_{0,i})$.

For $1 \leq g_1 \leq g-1$, let $\Delta_{g_1}[\bar \calt_{\ell, \bar{\gamma}}]$ be the image of 
$\kappa_{g_1,g_2}$ in $\bar \calt_{\ell, \bar{\gamma}}$, where
$g_2=g-g_1$ and $(\gamma_1, \gamma_2)$ ranges over the appropriate admissible pairs of class vectors.

\subsubsection{Clutching maps (non-compact type)}

In this case, let $g=g_1+g_2 + (\ell-1)$ and $n=n_1+n_2$ and 
$\gamma = (\gamma_1(1), \ldots, \gamma_1(n_1), \gamma_2(1), \ldots, \gamma_2(n_2))$.
The other clutching maps are:
\[\lambda_{g_1,g_2}: \bar \calt_{\ell, \bar{\gamma}_1;1}\times \bar \calt_{\ell, \bar{\gamma}_2;1} \to  
\bar \calt_{\ell, \bar{\gamma}}.\]
To define $\lambda_{g_1,g_2}$, consider a $\ZZ/\ell \ZZ$-curve $(Y_i/S, \iota_{0,i})$ with a mark $P_i$, for $i=1,2$.
One can glue these curves together to form a curve $Y/S$
by identifying the orbits of $P_1$ and $P_2$ in $\ell$ ordinary 
double points; (Specifically, identify $\iota_{0,1}(g)(P_1)$ and $\iota_{0,2}(g)(P_2)$ for $g \in G$).
Then $Y/S$ admits a $G$-action $\iota_0$ and has inertia type $\bar{\gamma}$.

Since $Y_1$ and $Y_2$ intersect in more than one point, the curve $Y/S$ has non-compact type.
By \cite[Ex.\ 9.2.8]{blr}, $\Pic^0(Y)$ is an extension
\begin{equation*}
0 \to Z \to  \Pic^0(Y)  \to \Pic^0(Y_1) \times \Pic^0(Y_2) \to 0,
\end{equation*}
where $Z$ is an $(\ell-1)$-dimensional torus.  Thus $Y$ has genus $g$
and the $p$-rank of $Y$ is:
\begin{equation}
\label{E:deltaif}
f(Y) = f(Y_1)+f(Y_2)+(\ell-1).
\end{equation}

There is an action of $\ZZ/\ell \ZZ$ on $Z$ and each of the non-trivial eigenspaces has dimension $1$;
we define the signature type of $Z$ to be $(1, \ldots, 1)$. 
The signature type of $(Y/S, \iota)$ is the sum of those for $(Y_i/S, \iota_{0,i})$ and $Z$.

For $0 \leq g_1 \leq g-(\ell-1)$, let $\Xi_{g_1}[\bar \calt_{\ell, g}] \subset \bar \calt_{\ell, g}$ be the image of 
$\lambda_{g_1,g_2}$, where $g_2=g-g_1-(\ell-1)$ and
$(\gamma_1, \gamma_2)$ ranges over the appropriate pairs of class vectors.
Let $\Delta_0[\bar \calt_{\ell, g}]$ 
be the union of $\Xi_{g_1}[\bar \calt_{\ell, g}]$ for $0 \leq g_1 \leq g-(\ell-1)$.
Then $\Delta_0[\bar \calt_{\ell, g}]$ is the set of moduli points of
stable $\ZZ/\ell \ZZ$-curves of genus $g$ which are not of compact type. 

\subsection{Components and dimension of the boundary}

The boundary of $\bar \calt_{\ell, g}$ is $\delta \calt_{\ell, g} = \bar\calt_{\ell, g} - \calt_{\ell, g}$.  
If $g \geq 2$, then $\delta \calt_{\ell, g}$ is the union of
$\Delta_i = \Delta_{i}[\bar\calt_{\ell, g}]$ for $1 \leq i \leq g-1$ 
and $\Xi_i = \Xi_i[\bar \calt_{\ell, g}]$ for $0 \leq i \leq g-(\ell-1)$, some of which may be empty. 
Note that $\Delta_i$ and $\Delta_{g-i}$ (resp.\ $\Xi_i$ and $\Xi_{g-i-(\ell-1)}$) 
are the same substack of $\bar\calt_{\ell, g}$.

If ${\mathcal S}$ is a stack with a map ${\mathcal S} \to \bar\calt_{\ell, g}$, let
$\Delta_i[{\mathcal S}]={\mathcal S} \times_{\bar\calt_{\ell, g}} \Delta_i[\bar\calt_{\ell,g}]$. 
So, $\Delta_i[\til\calt_{\ell, g}] =
\til\calt_{\ell, g}\times_{\bar\calt_{\ell, g}} \Delta_i$.  
Similar notation is used for $\Xi_i$. 

\begin{lemma} \label{lem:divisorboundary}
Every irreducible component of $\del\bar\calt_{\ell, g}$ has dimension ${\rm dim}(\calt_{\ell, g})-1$.
\end{lemma}

\begin{proof}
Let $W$ be an irreducible component of $\delta \calt_{\ell, g}$.
There is an inertia type $\bar \gamma$ such that $W$ is either a component of 
(i) $\Delta_i[\calt_{\ell, \bar \gamma}]$ for some $0 \leq i \leq g-1$
or (ii) $\Xi_i[\calt_{\ell, \bar \gamma}]$ for some $0 \leq i \leq g-(\ell-1)$.
By Lemma~\ref{Lmoduli}(5), it suffices to show that ${\rm dim}(W)=n-4$.

Case (i): 
In this case, a generic point of $W$ is the moduli point of a singular curve $Y$ with two irreducible components $Y_1$ and $Y_2$ intersecting in one ordinary double point $y$.
Let $\bar \gamma_i$ be the inertia type of the restriction of the $\ZZ/\ell \ZZ$-action to $Y_i$
and let $n_i=n(\bar \gamma_i)$.  
Then $n_1+n_2 -2 = n$ since $y$ is a ramification point for the two restrictions. So
\[{\rm dim}(W)=\dim(\calt_{\ell, \bar \gamma_1}) + \dim(\calt_{\ell, \bar \gamma_2})=(n_1-3)+(n_2-3)=n_1+n_2-6=n-4.\] 

Case (ii):
In this case, a generic point of $W$ is the moduli point of a singular curve $Y$ with two irreducible components $Y_1$ and $Y_2$, of genera $i$ and $g-i-(\ell-1)$ intersecting at one unramified $\ZZ/\ell \ZZ$-orbit.  
Let $\bar \gamma_i$ be the inertia type of the restriction of 
the $\ZZ/\ell \ZZ$-action to $Y_i$ and let $n_i=n(\bar \gamma_i)$.  
Then $n_1+n_2=n$.
There is a 1-dimensional choice of an unramified orbit on each of $Y_1$ and $Y_2$.
So 
\[{\rm dim}(W)={\rm dim}(\calt_{\ell, \bar \gamma_1}) + 1 + {\rm  dim}(\calt_{\ell, \bar \gamma_2}) + 1= (n_1-3)+(n_2-3)+2 =n-4.\]

\end{proof}

The next result will be used to find an upper bound on the dimension of the $p$-rank strata.

\begin{proposition}\label{prop:upperbound} 
If ${\mathcal S} \subset \bar \calt_{\ell,g}$ has the property that ${\mathcal S}$ intersects $\Delta_i$, then
\[{\rm dim}({\mathcal S}) \leq {\rm dim}(\Delta_i[{\mathcal S}]) +1.\]
\end{proposition}

\begin{proof}
A smooth proper stack has the same intersection-theoretic properties as a 
smooth proper scheme \cite[p.\ 614]{V:stack}.
In particular, if two closed substacks of $\bar \calt_g$ intersect 
then the codimension of their intersection is at most the sum of their codimensions.
Now $\Delta_i[\bar \calt_{\ell,g}]$ is a closed substack of $\bar \calt_{\ell,g}$.
It suffices to consider the case that ${\mathcal S}$ is closed.
Thus 
\[{\rm codim}(\Delta_i[{\mathcal S}], \bar \calt_{\ell,g}) \leq 
{\rm codim}(\Delta_i, \bar \calt_{\ell,g}) + {\rm codim}({\mathcal S}, \bar \calt_{\ell,g}).\]
The result follows from Lemma~\ref{lem:divisorboundary} since ${\rm codim}(\Delta_i, \bar \calt_{\ell, g})=1$.
\end{proof}

\subsection{The $p$-rank stratification} 

If $A$ is a
semi-abelian scheme over a Deligne-Mumford stack ${\mathcal S}$, then there is a
stratification ${\mathcal S} = \cup {\mathcal S}^{f}$ by locally closed reduced substacks such that $s \in {\mathcal S}^f(k)$ if and only if $f(A_s) =f$, \cite[Theorem 2.3.1]{katzsf}, see also \cite[Lemma 2.1]{AP:monoprank}. For example, ${\calt}_{\ell, g}^f$ is the locally closed reduced substack of ${\calt}_{\ell,g}$
whose points represent smooth $\ZZ/\ell \ZZ$-curves of genus $g$ with $p$-rank $f$.

We use the following notation for the $p$-rank $f$ stratum of the boundary, $\Delta_i[\bar\calt_{\ell, g}]^f:=(\Delta_i[\bar\calt_{\ell, g}])^f$. 
These strata are easy to describe using the clutching maps.  
First, if $1 \leq i \leq g-1$, then \eqref{eqblrprank} implies that $\Delta_i[\bar \calt_{\ell, g}]^f$ is the union of the images of $\til \calt_{\ell, i}^{f_1} \times \til \calt_{\ell, g-i}^{f_2}$ 
under $\kappa_{i, g-i}$ as $(f_1,f_2)$ ranges over all pairs (satisfying Lemma~\ref{lem:edivf}) such that
\begin{equation*}
0 \leq f_1 \leq i,\ 0 \leq f_2 \leq g-i \text{ and } f_1+f_2=f.
\end{equation*}
Second, if $f \geq 2$ and $0 \leq i \leq g-(\ell-1)$, then \eqref{E:deltaif} implies that $\Xi_i[\bar \calt_{\ell, g}]^f$ is the union of the images of 
$\bar \calt_{\ell, i;1}^{f_1} \times \bar \calt_{\ell, g-(\ell-1)-i;1}^{f_2}$ under $\lambda_{i, g-(\ell-1)-i}$ 
as $(f_1,f_2)$ ranges over all pairs (satisfying Lemma~\ref{lem:edivf}) such that
\begin{equation*}
\label{Xif1f2conditions}
0 \leq f_1 \leq i,\ 0 \leq f_2 \leq g-(\ell-1)-i \text{ and } f_1+f_2=f-(\ell-1).
\end{equation*}

\subsection{Shimura varieties} \label{Sshimura}

We briefly review some notation about Shimura varieties that we need in Sections~\ref{Sbasecases} and \ref{Sapplication2}.  We refer to \cite[Section 3.3]{LMPT2} for a longer explanation.

\begin{notation} \label{Ntorelli}
Let $\ell$ be an odd prime. 
Consider an inertia type $\bar \gamma$ for $\ell$.
It determines the number of branch points $n=n(\bar \gamma)$ and 
the genus $g=g(\bar \gamma)$ as in \eqref{Egenus} 
for a $\ZZ/\ell\ZZ$-cover $Y \to {\mathbb P}^1$ with inertia type $\bar \gamma$.
Furthermore, it determines the signature type of the cover as in Lemma~\ref{Lsignature}. 
\end{notation}

Recall that $\bar \calt_{\ell, {\bar \gamma}}$ is the moduli space of $\ZZ/\ell\ZZ$-covers $Y \to {\mathbb P}^1$
with inertia type $\bar{\gamma}$.
By Lemma~\ref{Lmoduli}, $\bar \calt_{\ell, {\bar \gamma}}$ is irreducible and has 
dimension $n(\bar{\gamma})-3$. 

\begin{notation} \label{Nshimura}
Let ${\mathcal A}_g$ be the moduli space of principally polarized abelian varieties of dimension $g$.
Consider the image of $\bar \calt_{\ell, {\bar \gamma}}$ in ${\mathcal A}_g$.
Let $Z_{\bar{\gamma}}=Z(\ell, n, \bar{\gamma})$ be the closure of this image; its points represent Jacobians of 
curves (smooth or of compact type) that admit a $\ZZ/\ell\ZZ$-cover of ${\mathbb P}^1$ 
with inertia type $\bar{\gamma}$.

Attached to the data of $\ell$ and the signature type, there is a PEL-type Shimura variety ${\rm Sh}$. 
Let $\Sigma_{\bar{\gamma}}=\Sigma(\ell, n, \bar{\gamma})$ be the irreducible component of ${\rm Sh}$ 
that contains $Z_{\bar{\gamma}}$.
\end{notation}


\section{Intersection of the $p$-rank $0$ stratum with the boundary}

In this section, we study the geometry of the $p$-rank stratification on the moduli space of cyclic degree $\ell$ 
covers of the projective line.

Recall that $p$ is a prime such that $p \not = \ell$ and
$e$ is the order of $p$ modulo $\ell$.  
The formula for the upper bound $B(\bar \gamma)$ for the $p$-rank of a cover with inertia type $\bar \gamma$ is in \eqref{Ebound}. 
Let $f$ be a multiple of $e$ such that $0 \leq f < B(\bar \gamma)$.
Define $\epsilon = 1$ if $p \equiv 1 \bmod \ell$ and $\epsilon=0$ otherwise.

We first give a lower bound on the dimension of the $p$-rank strata.

\begin{proposition} \label{Tpurity} 
Suppose the $p$-rank $f$ stratum $\bar \calt^f_{\ell,\bar \gamma}$ is non-empty and let 
${\mathcal S}$ be an irreducible component of it.
Then 
\begin{equation} \label{Elowerbound}
{\rm dim}({\mathcal S}) \geq {\rm dim}(\bar \calt_{\ell, \bar \gamma}) - (B(\bar \gamma)-f)/e + \epsilon.
\end{equation}
\end{proposition}

\begin{proof}
The $p$-ranks which occur on $\bar \calt_{\ell,\bar \gamma}$ are multiples of $e$ by Lemma~\ref{lem:edivf} and 
are at most $B(\bar \gamma)$ by Theorem~\ref{LboundB}.  Also, if $p \equiv 1 \bmod \ell$, then $e=1$ and 
$f \not = g-1 = B(\bar \gamma) -1$ by Lemma~\ref{lem:edivf}.
So the number of integers $f'$ such that $f < f' \leq B(\bar \gamma)$ which can 
occur as the $p$-ranks for points of $\bar \calt_{\ell,\bar \gamma}$ is at most $(B(\bar \gamma)-f)/e + \epsilon$.
The statement is then an immediate application of the purity result of Oort \cite[Lemma 1.6]{Oort} which states that if the $p$-rank changes, then it does so on a subspace of codimension $1$.
\end{proof}

\begin{remark}
For $\ell \geq 5$, the lower bound on the right hand side of \eqref{Elowerbound} is positive only
when $f$ is large relative to $g$.
For example, if $\ell =5$ and $p \equiv 1 \bmod 5$, then it is $-g/2 + f$.
\end{remark}

In the next result, assuming that the $p$-rank $0$ stratum $\bar \calt_{\ell, {\bar \gamma}}^0$ is non-empty, 
we show that it
intersects the boundary deeply 
(in the sense that the intersection contains points corresponding to reducible curves with many components). 

\begin{theorem}\label{Cbreakupf=0}
Suppose ${\mathcal S}$ is an irreducible component of the $p$-rank $0$ stratum
$\bar \calt_{\ell, {\bar \gamma}}^0$ of $\bar \calt_{\ell, {\bar \gamma}}$.
Let $\sigma = {\rm dim}({\mathcal S})$.
Then there exists $\eta \in {\mathcal S}$ such that the curve $Y_\eta$ of compact type represented by $\eta$ is reducible,
with at least 
$\sigma + 1$ components,
such that the $\ZZ/\ell \ZZ$-action stabilizes and acts non-trivially on each component.
\end{theorem}

Before proving this theorem, we explain its significance.

\begin{remark} \label{Rsignificance}
\begin{enumerate}
\item Recall that
$\sigma \geq {\rm dim}(\bar \calt_{\ell, {\bar \gamma} }) - B(\bar{\gamma})/e + \epsilon$ by Proposition~\ref{Tpurity}. 
So, for each irreducible component of the $p$-rank $0$ stratum, Theorem~\ref{Cbreakupf=0} guarantees the existence
of a point representing a curve that is reducible, with many components.  
The existence of a degenerate point 
of this type can be helpful for studying the $p$-rank $0$ strata.
We illustrate this with several applications in Section~\ref{Sapplication2}.

\item Theorem~\ref{Cbreakupf=0} is a generalization of \cite[Theorem 3.11(c)]{APhypmono}, which is the case $\ell=2$.
Suppose $\ell=2$ and $g \geq 2$, in which case there is a unique inertia type $\bar{\gamma}$ for hyperelliptic 
curves of genus $g$.
In this case, Theorem~\ref{Cbreakupf=0} applies to an irreducible component 
${\mathcal S}$ of the $p$-rank $0$ stratum of the locus of hyperelliptic curves of genus $g$. 
By \cite[Proposition 2]{GP}, $\sigma = {\rm dim}({\mathcal S}) = g-1$. 
So Theorem~\ref{Cbreakupf=0} shows that ${\mathcal S}$ contains a point representing a curve that has $g$ components (each of which has genus $1$ and is thus a supersingular elliptic curve);
this is the conclusion of \cite[Theorem 3.11(c)]{APhypmono}.

\item In contrast, when $\ell$ is odd, then usually $\sigma < g-1$.  
Thus Theorem~\ref{Cbreakupf=0} does not imply that 
$\bar \calt_{\ell, {\bar \gamma}}^0$ contains a point representing a reducible curve with $g$ components.  
This makes it harder to study the case when $\ell$ is odd.

\item It is not possible to prove Theorem~\ref{Cbreakupf=0} using results on the boundary of the moduli space of $n$-marked curves of genus $0$.  The reason is that the relationship between the $p$-rank and the location of the 
branch points is extremely complicated.  As an example of this, see the case that $\ell=3$ and $g=2$ studied in Lemma~\ref{Lcalculateg2}.  In other words, it is not clear how to maintain the $p$-rank $0$ condition 
when deforming the curve by moving the branch points.  
\end{enumerate}
\end{remark}

\begin{proof} (Proof of Theorem~\ref{Cbreakupf=0})
The proof is by induction on the number of branch points $n=n(\bar \gamma)$.  
This is equivalent to induction on the genus $g=g(\bar \gamma)$, because $g=(n-2)(\ell-1)/2$.
For the base case, when $n=3$ and $g=(\ell-1)/2$, the statement is vacuous since $\calt_{\ell, {\bar \gamma} }$ has dimension $0$.

Suppose that the statement is true for all inertia types $\bar \gamma'$ for which the genus $g'$ is less than $g$.
Let $\bar \gamma$ be an inertia type for which the genus is $g$ and let 
${\mathcal S}$ be an irreducible component of $\bar \calt_{\ell, {\bar \gamma}}^0$.
When $\sigma=0$, the statement is vacuous.  

Suppose $\sigma > 0$.  Since $\calt_{\ell, {\bar \gamma}}$ is affine,
${\mathcal S}$ intersects a boundary component of $\bar \calt_{\ell, {\bar \gamma}}$.  
The points of ${\mathcal S}$ represent curves whose $p$-rank is $0$, hence 
\eqref{E:deltaif} implies that ${\mathcal S}$ does not intersect $\Delta_{0}$.
Thus ${\mathcal S}$ intersects $\Delta_j$ for some $1 \leq j \leq g-1$.  
A point of $\Delta_j[{\mathcal S}]$ represents a curve having at least two components 
(which completes the proof when $\sigma=1$).

By Proposition~\ref{prop:upperbound}, ${\rm dim}(\Delta_j[{\mathcal S}]) \geq \sigma -1$.
A point $\eta_0$ of $\Delta_j[{\mathcal S}]$ is in the image of a clutching morphism.
Specifically, there is an admissible pair of inertia types ${\bar \gamma} _1, {\bar \gamma} _2$, 
and points $\xi_i \in \tilde{\calt}_{\ell, {\bar \gamma} _i}$, for $i=1,2$, such that 
$\eta_0=\kappa_{j, g-j}(\xi_1, \xi_2)$.  
Since ${\mathcal S}$ is an irreducible component of $\bar \calt_{\ell, {\bar \gamma}}^0$, there
is an irreducible component $\Gamma_i$ of $\tilde{\calt}_{\ell, {\bar \gamma}_i}$, for $i=1,2$, such that 
$\kappa_{j, g-j}(\Gamma_1, \Gamma_2) \subset \Delta_j[{\mathcal S}]$.

Note that $g({\bar \gamma} _i) <g$.
Let $\sigma_i={\rm dim}(\Gamma_i)$.  Then $\sigma_1+\sigma_2 \geq \sigma-1$.  
By the inductive hypothesis, for $i=1,2$, there exists $\eta_i \in \Gamma_i$ such that the curve $Y_{\eta_i}$ of compact type represented by $\eta_i$ has at least $\sigma_i+1$ components.
Then $\kappa_{j, g-j}(\eta_1, \eta_2) \in \Delta_j[{\mathcal S}]$ has at least $(\sigma_1+1)+(\sigma_2+1) \geq \sigma+1$ components. 
\end{proof}

\section{Trielliptic covers} \label{Strielliptic}

In this section, we specialize to the case $\ell = 3$.
Suppose $p \not = 3$ is prime.
We study the $p$-ranks of trielliptic curves, which are $\ZZ/3 \ZZ$-covers of ${\mathbb P}^1$.
Suppose $g \geq 2$ and $(r,s)$ is a signature type for $g$.

In Proposition~\ref{Cf=B-2}, for all primes $p \geq 5$,  
we prove that there exists 
a trielliptic curve $Y$ defined over $\bar{\FF}_p$ of genus $g$ and signature $(r,s)$
whose $p$-rank is smaller than the upper bound $B(r,s)$. 

In Theorem~\ref{Trsexistence}, when $p \equiv 2 \bmod 3$ is odd, 
we prove that every integer $f$ satisfying the necessary conditions from Lemma~\ref{lem:edivf}
occurs as the $p$-rank of a trielliptic curve $Y$ defined over $\bar{\FF}_p$ of genus $g$ and
signature $(r,s)$;
in addition, we prove that there is an irreducible component of $\calt_{g, (r,s)}^f$ whose dimension 
equals the lower bound from Proposition~\ref{Tpurity}.

\subsection{Notation for trielliptic covers}

Suppose $(Y/S, \iota_0)$ is a smooth trielliptic curve.  The $\ZZ/3 \ZZ$-cover $\psi:Y \to \Proj^1$ 
has an equation of the form:
\begin{equation} \label{Etrielliptic}
y^3=\prod_{i=1}^{d_1} (x-\alpha_i) \prod_{i=1}^{d_2} (x-\beta_i)^2 .
\end{equation}
Without loss of generality, we assume that $\psi$ is not branched at $\infty$.
The number of branch points of $\psi$ is $n=d_1+d_2$ and the genus of $Y$ is $g=d_1+d_2-2$.
The {\em inertia type} of $\psi$ is ${\bar \gamma} =(\underbrace{1,\ldots,1}_{d_1},\underbrace{2,\ldots,2}_{d_2})$.

\begin{lemma} \label{Lgamma} \cite[Lemma 2.7]{AP:tri}
The set of inertia types ${\bar \gamma} $ for a trielliptic curve $(Y/S, \iota_0)$ of genus $g$
is in bijection with 
$\{(d_1,d_2) \mid d_1,d_2 \in \ZZ^{\geq 0}, \ d_1+d_2=g+2, \ d_1+2d_2 \equiv 0 \bmod 3\}$.
\end{lemma}

There is a $\ZZ/3 \ZZ$-eigenspace decomposition $H^0(Y, \Omega^1_Y)= \LL_1 \oplus \LL_2$ where 
$\omega \in \LL_i$ if $\zeta_3 \circ \omega = \zeta_3^i \omega$.
The {\em signature type} of $(Y/S, \iota_0)$ is $(r,s)$ where $r=\dim(\LL_1)$ and $s=\dim(\LL_2)$. 

If $(Y/S, \iota_0)$ is a trielliptic curve then so is $(Y/S, \iota'_0)$  
where $\iota'_0(1)=\iota_0(2)$. 
Replacing $\iota_0$ with $\iota'_0$ exchanges the values of $d_1$ and $d_2$ and the values of $r$ and $s$.

\begin{definition} \label{Dtrisig}
A {\em trielliptic signature} for $g \in \NN$ is a pair $(r,s)$ of integers with 
$r+s=g$, and $0 \leq {\rm max}\{r,s\} \leq 2 {\rm min}\{r,s\}+1$.
\end{definition}

The next result follows from Lemma~\ref{Lsignature}

\begin{lemma} \label{L:bijection} 
There is a bijection between trielliptic signatures $(r,s)$ for $g$ 
and inertia types of $\ZZ/3 \ZZ$-Galois covers of $\Proj^1$ of genus $g$ 
given by the formulae \[d_1=2r-s+1, \ d_2=2s-r+1.\]
\end{lemma}

In other words, $r=(2d_1+d_2-3)/3$ and $s=(d_1+2d_2-3)/3$.

\begin{example} \label{Eg=1}
(Signature $(1,0)$, inertia type ${\bar \gamma} =(1,1,1)$).
There is a unique smooth elliptic curve which is trielliptic. 
It  has $p$-rank $0$ when $p \equiv 2 \bmod 3$ and $p$-rank $1$ when $p \equiv 1 \bmod 3$.
\end{example}

\begin{proof}
By \cite[Theorem 10.1]{silverman}, an elliptic curve with automorphism of order $3$ has $j$-invariant $0$.
The result follows from \cite[Exercise V.5.7, Example V.4.4]{silverman}.
\end{proof}

\subsection{Components of the moduli space and maximal $p$-rank}

Let $(r,s)$ be a trielliptic signature for $g$.  Let $\bar \gamma$ be the inertia type given by
$\bar \gamma(i) = d_i$ where $d_i$ are as in Lemma~\ref{L:bijection}.
Let $f$ be an integer $0 \leq f \leq g$ satisfying the conditions of Lemma~\ref{lem:edivf}, 
namely, $f$ is even if $p \equiv 2 \bmod 3$ and $f \not = g-1$ if  $p \equiv 1 \bmod 3$.

Consider the moduli space $\calt_{(r,s)}=\calt_{3, \bar{\gamma}}$ 
of smooth trielliptic curves with signature $(r,s)$
and inertia type $\bar \gamma$.
For $f$ as above,
let $\calt_{(r,s)}^f$ denote the $p$-rank $f$ stratum of $\calt_{(r,s)}$.
Similarly, define $\bar{\calt}^f_{(r,s)}$ by replacing the word smooth by stable.

The next result is a special case of Proposition~\ref{Tpurity}.

\begin{proposition} \label{Tpurity3} 
Suppose ${\mathcal S}$ is an irreducible component of $\bar \calt_{(r,s)}^f$.
If $p \equiv 2 \bmod 3$, then ${\rm dim}({\mathcal S}) \geq {\rm max}\{r,s\} -1 + f/2$.
If $p \equiv 1 \bmod 3$ and $f<g$, then ${\rm dim}({\mathcal S}) \geq f$.
\end{proposition}

We first consider the case of maximal $p$-rank.
Define $B(r,s)=g$ if $p \equiv 1 \bmod 3$ and $B(r,s)=2\min\{r,s\}$ if $p \equiv 2 \bmod 3$.
By Theorem~\ref{LboundB}, the $p$-rank of a trielliptic curve of signature $(r,s)$ satisfies $f \leq B(r,s)$.

\begin{proposition} \label{Pf=B} [Bouw]
If $p \not = 3$, then there exists a smooth trielliptic curve with signature $(r,s)$ and $p$-rank 
$f_{\rm max}:=B(r,s)$.
The $p$-rank $f_{\rm max}$ strata $\calt_{(r,s)}^{f_{\rm max}}$ 
is open and dense in $\calt_{(r,s)}$.
\end{proposition}

\begin{proof}
The first statement is a special case of \cite[Propositions 7.4, 7.8]{Bouw}. 
The second statement follows since $\calt_{(r,s)}$ is irreducible and the $p$-rank is lower 
semi-continuous.
\end{proof}

\subsection{Base cases} \label{Sbasecases}

Moonen proved there are exactly 20 families of cyclic covers of ${\mathbb P}^1$ that are \emph{special}, 
meaning that the image of the family under the Torelli morphism is open and dense in the 
associated PEL type Shimura variety;  
these are listed as $M[1]-M[20]$ in \cite[Table 1]{moonen}.
In \cite[Sections 4-6]{LMPT2}, the authors computed the Newton polygons occurring on these families.
In \cite[Theorem 5.11]{LMPT2} and \cite[Theorem 7.1]{LMPT3}, they proved that each of these Newton polygons occurs 
for the Jacobian of a smooth curve in the family, except possibly the supersingular ones when $p$ is small.  

For trielliptic covers, there are three families that are special: $M[3]$, $M[6]$, and $M[10]$. 
Since the $p$-rank is an invariant of the Newton polygon, 
we can find the dimension of the $p$-rank strata of these families.
When the Newton polygon is supersingular (which happens only when $f=0$), we can 
remove the requirement that $p>>0$ in all but one case.

The results below for the signature $(r,s)$ are also true for the signature $(s,r)$. 

\begin{lemma} \label{Lbasecase}
\begin{enumerate}
\item $M[3]$ (Signature $(1,1)$, inertia type $\bar \gamma =(1,1,2,2)$).

If $p \geq 5$ and $f=0$, then $\calt_{(1,1)}^0$ is non-empty of dimension $0$.

\item $M[6]$ (Signature $(2,1)$, inertia type $\bar \gamma =(1,1,1,1,2)$).

If $p \equiv 1 \bmod 3$ and $f = 0,1$, then $\calt_{(1,2)}^f$ is non-empty of dimension $f$.

If $p \equiv 2 \bmod 3$ and $f=0$, then $\calt_{(1,2)}^f$ is non-empty of dimension $1$.

\item $M[10]$ (Signature $(3,1)$, inertia type $\bar \gamma =(1,1,1,1,1,1))$.

If $p \equiv 1 \bmod 3$ and $f=0,1,2$, then $\calt_{(1,3)}^f$ is non-empty of dimension $f$ (if $p>>0$ when $f=0$).

If $p \equiv 2 \bmod 3$ and $f=0$, then $\calt_{(1,3)}^0$ is non-empty of dimension $2$.
\end{enumerate}
\end{lemma}

\begin{proof}
The Newton polygons for the curves in the family are listed on \cite[page 19]{LMPT2}.
\begin{enumerate}
\item When $f=0$, then the Newton polygon of a curve in the family is supersingular.   
If $p \equiv 1 \bmod 3$, the result follows from \cite[Theorem 5.11]{LMPT2}.
The main idea is that, because of Example~\ref{Eg=1}, a supersingular 
curve in the family must be smooth.

When $p \equiv 2 \bmod 3$ is odd, the result follows from Lemma~\ref{Lcalculateg2} below
(or \cite[Theorem 7.1]{LMPT3} when $p>>0$).

\item When $p \equiv 1 \bmod 3$, the result follows from \cite[Theorem 5.11]{LMPT2};
note that the Newton polygon has slopes $1/3$ and $2/3$ when $f=0$. 

When $p \equiv 2 \bmod 3$, the result follows from \cite[Theorem 7.1]{LMPT3} when $p>>0$.
Here is an argument that removes the hypothesis $p >>0$.
Let $S$ be an irreducible component of the $p$-rank $0$ locus $\bar \calt_{(2,1)}^0$.  
Then ${\rm dim}(S)=1$ because $\bar \calt_{(2,1)}$ has dimension $2$ 
and its generic geometric point represents a curve with $p$-rank $2$.
The intersection of $S$ with the boundary is contained in $\kappa_{1,2}(\til{\calt}_{(1,0)}^0 \times \til{\calt}_{(1,1)}^0)$, 
but that only has dimension $0$, so the generic geometric point of $S$ represents a smooth curve. 

 \item 
When $p \equiv 1 \bmod 3$, the result follows from \cite[Theorem 7.1]{LMPT3};
we do not know how to remove the hypothesis $p>>0$ when $f=0$. 
 
When $p \equiv 2 \bmod 3$, then $\calt_{(1,3)}^0$ is non-empty of dimension $2$
by \cite[Theorem 5.11]{LMPT2}.
\end{enumerate}
\end{proof}

For the family $M[10]$, 
we consider the PEL type Shimura variety ${\rm Sh}$ attached to the data of $\ell =3$ 
and the signature type $(3,1)$.  As in Section~\ref{Sshimura}, 
let $\Sigma$ be the irreducible component of ${\rm Sh}$ which contains the Torelli locus.

\begin{proposition}\label{Pirreducible}
Suppose $p \equiv 2 \bmod 3$.
For the family $M[10]$, 
the $p$-rank $0$ stratum $\Sigma^0$ of $\Sigma$ is irreducible
and thus $\calt_{(1,3)}^0$ is irreducible.
\end{proposition}

\begin{proof}
For each generic geometric point of $\Sigma^0$, we consider the Newton polygon 
$\nu$ of the abelian variety represented by this point.
Applying the Kottwitz method, see \cite[Section 4.3 and table on page 19]{LMPT2}, 
shows that $\nu$ has slopes $1/4$ and $3/4$ when $p \equiv 2 \bmod 3$; 
in particular, it is not supersingular.
The hypotheses of \cite[Theorem~1.1]{achter14} are satisfied;
the conclusion of that result is that the stratum of $\Sigma$ with Newton polygon $\nu$ 
is irreducible.
Since this stratum is open and dense in $\Sigma^0$, this implies that $\Sigma^0$ is irreducible.

Since the family $M[10]$ is special, the image of $\calt_{(3,1)}$ is open and dense in $\Sigma$.
It follows that $\calt_{(3,1)}^0$ is open and dense in $\Sigma^0$.
Thus $\calt_{(3,1)}^0$ is irreducible as well. 
\end{proof}

\subsection{Trielliptic curves whose $p$-rank is not maximal}

The next result extends Proposition~\ref{Pf=B} by 
showing for each prime $p \geq 5$, that there exist trielliptic curves of each signature type $(r,s)$
whose $p$-rank is not the maximum $B(r,s)$.
Recall that $B(r,s) = 2 {\rm min}\{r,s\}$ when $p \equiv 2 \bmod 3$ 
and $B(r,s)=r+s$ when $p \equiv 1 \bmod 3$. 

\begin{proposition} \label{Cf=B-2}
Let $p \geq 5$ and $g \geq 2$.
Let $(r,s)$ be a trielliptic signature for $g$. 
Then $\calt_{(r,s)}^{B(r,s)-2}$ is non-empty and each of its irreducible components has 
dimension $g-2$ (codimension $1$ in $\calt_{(r,s)}$).
Thus there exists a smooth trielliptic curve with signature $(r,s)$ and $p$-rank $f=B(r,s)-2$.
\end{proposition}

\begin{proof} 
Recall that ${\rm dim}(\calt_{(r,s)})=g-1$ and the generic geometric point of $\calt_{(r,s)}$ 
represents a trielliptic curve with $p$-rank $B(r,s)$ by Proposition~\ref{Pf=B}.
If $\calt_{(r,s)}^{B(r,s)-2}$ is non-empty, then, by definition, each of its generic geometric points represents 
a smooth trielliptic curve with signature $(r,s)$ and $p$-rank $f=B(r,s)-2$.
Furthermore, if ${\mathcal S}$ is one of the irreducible components of $\calt_{(r,s)}^{B(r,s)-2}$, 
then ${\rm dim}({\mathcal S}) \leq g-2$ and Proposition~\ref{Tpurity3} 
implies that ${\rm dim}({\mathcal S}) \geq g-2$, 
so ${\rm dim}(S)=g-2$.

It thus suffices to prove that $\calt_{(r,s)}^{B(r,s)-2}$ is non-empty.
Without loss of generality, suppose $r \leq s$.
The proof is by induction on $r$.

If $r=1$, then $1 \leq s \leq 3$.
If $r=1$ and $s=1$, then $B(1,1)-2=0$ and the result follows from Lemma~\ref{Lbasecase}(1) 
(deferred to Lemma~\ref{Lcalculateg2}(1) when $p \equiv 2 \bmod 3$).
If $r=1$ and $s=2$, then $B(1,2)-2=1$ when $p \equiv 1 \bmod 3$ and $B(1,2)-2=0$ when $p \equiv 2 \bmod 3$
and the result follows from Lemma~\ref{Lbasecase}(2).
If $r=1$ and $s=3$, then $B(1,3)-2=2$ when $p \equiv 1 \bmod 3$ and $B(1,3)-2=0$ when $p \equiv 2 \bmod 3$
and the result follows from Lemma~\ref{Lbasecase}(3).

Now suppose $2 \leq r \leq s$.

\smallskip

{\bf Case 1:} Suppose $s \leq 2r$.  Then $(r-1,s-1)$ is a valid trielliptic signature.
Note that $B(r-1,s-1)=B(r,s)-2$.
Let ${\mathcal S}_1$ be an irreducible component of $\calt_{(1,1)}^0$, 
which is non-empty when $p \geq 5$ by Lemma~\ref{Lbasecase}(1).
Let ${\mathcal S}_2$ be an irreducible component of $\calt_{(r-1,s-1)}^{B(r-1,s-1)}$, which is non-empty by Proposition~\ref{Pf=B}.
Consider an irreducible component $\til {\mathcal S}_1$ of $\til \calt_{(1,1)}^0$ lying above ${\mathcal S}_1$
and an irreducible component $\til {\mathcal S}_2$ of $\til \calt_{(r-1,s-1)}^{B(r-1,s-1)}$ lying above ${\mathcal S}_2$.

When $(r,s)=(1,1)$, then $d_1=d_2=2$ are both positive.
Thus without loss of generality, we can choose $\til {\mathcal S}_1$ (the labeling of the ramification points) 
so that the clutching situation below is admissible:
\[\kappa_{2, g-2}:\til{{\mathcal S}}_1 \times \til{{\mathcal S}}_2 \to \bar \calt_{(r,s)}^{B(r,s) - 2}.\]
Let $K=\kappa_{2, g-2}(\til{{\mathcal S}}_1 \times \til{{\mathcal S}}_2)$.
By construction, $K$ is contained in $\Delta_2[\bar \calt_{(r,s)}^{B(r,s) - 2}]$.

Let $W$ be an irreducible component of $\bar \calt_{(r,s)}^{B(r,s) - 2}$ which contains $K$.
By the same reasoning as the first paragraph of the proof, ${\rm dim}(W) = g-2$.
On the other hand, since ${\rm dim}(\til{{\mathcal S}}_1)=0$ and ${\rm dim}(\til \calt_{(r-1,s-1)}) = g-3$, 
it follows that 
\[{\rm dim}(K) = {\rm dim}(\til{{\mathcal S}}_1) + {\rm dim}(\tilde{\calt}_{(r-1,s-1)}) = g-3.\]
Thus the generic point of $W$ is not contained in $K$.

By construction, the generic point of ${\mathcal S}_1$ represents a smooth curve.
The generic point of $\bar{\calt}_{(r-1,s-1)}^{B(r-1,s-1)}$ represents a smooth curve by Proposition~\ref{Pf=B}
and Lemma~\ref{lem:divisorboundary}.
So the generic point of $W$ is not contained in any other boundary component.
Thus the generic point of $W$ represents a smooth curve
and $\calt_{(r,s)}^{B(r,s)-2}$ is non-empty, with irreducible components of dimension $g-2$.

\smallskip

{\bf Case 2:} Suppose $s = 2r+1$.   Then $(r-1,s-2)$ is a valid trielliptic signature.
Note that $B(r-1,s-2)=2(r-1)$ when $p \equiv 2 \bmod 3$ and $B(r-1,s-2)=g-3$ when $p \equiv 1 \bmod 3$.
Let $f'=0$ when $p \equiv 2 \bmod 3$ and $f'=1$ when $p \equiv 1 \bmod 3$.
Then $f'+ B(r-1,s-2)=B(r,s)-2$.

Let ${\mathcal S}_1$ be an irreducible component of $\calt_{(1,2)}^{f'}$, 
which is non-empty by Lemma~\ref{Lbasecase}(2).
Let ${\mathcal S}_2$ be an irreducible component of $\calt_{(r-1,s-2)}^{B(r-1,s-2)}$, 
which is non-empty by Proposition~\ref{Pf=B}.

When $(r,s)=(1,2)$, then $d_1=1$ and $d_2=4$, which are both positive.
We repeat the argument above, making an admissible clutching of the following form:
\[\kappa_{3, g-3}:\til{{\mathcal S}}_1 \times \til{{\mathcal S}}_2 \to \bar \calt_{(r,s)}^{B(r,s) - 2}.\]
The rest of the proof is the same.
\end{proof}

\begin{remark}
When $p=2$, then $\calt_{(1,1)}^0$ is empty, as shown in Lemma~\ref{Lcalculateg2}(2).
When $p=2$, it is still true that $\bar{\calt}_{(1,1)}^0$ is non-empty (of dimension $0$); 
the same proof as for Proposition~\ref{Cf=B-2} 
shows that $\bar{\calt}_{(r,s)}^{B(r,s)-2}$ is non-empty and each of its irreducible components has 
dimension $g-2$ (codimension $1$ in $\bar{\calt}_{(r,s)}$), but it is not clear whether any of its 
points represents a smooth curve.

The case when $p=3$ is described in Proposition~\ref{Pwild3}.
\end{remark}

\subsection{Existence of trielliptic curves with given $p$-rank}

In this section, suppose $p \equiv 2 \bmod 3$. 
In this case, the necessary conditions on the $p$-rank
are that $f$ is even and $0 \leq f \leq 2{\rm min}(r,s)$.   
For every signature type and for all odd $p \equiv 2 \bmod 3$, 
we prove that every $p$-rank $f$ satisfying the necessary conditions occurs for a smooth 
trielliptic curve of that signature in characteristic $p$.
If ${\mathcal S}$ is an irreducible component of $\bar \calt_{(r,s)}^f$, recall from Proposition~\ref{Tpurity3} 
that ${\rm dim}({\mathcal S}) \geq {\rm max}\{r,s\} -1 + f/2$.  

\begin{theorem} \label{Trsexistence}
Let $p \equiv 2 \bmod 3$ be odd and let $g \geq 2$.  Let $(r,s)$ be a trielliptic signature for $g$.
Suppose $0 \leq f \leq 2{\rm min}\{r,s\}$ is even.
Then there exists a smooth trielliptic curve of genus $g$ defined over $\bar{\FF}_p$ 
with signature type $(r,s)$ and $p$-rank $f$.
More generally, $\calt_{(r,s)}^f$ is non-empty and contains an irreducible component ${\mathcal S}={\mathcal S}_{(r,s)}^f$ 
with $\dim({\mathcal S}) = \max\{r,s\}-1 +f/2$. 
\end{theorem}

\begin{proof}
The first statement about the existence of the trielliptic curve 
with signature type $(r,s)$ and $p$-rank $f$ is equivalent to the statement that $\calt_{(r,s)}^f$ is non-empty.

To prove this, without loss of generality, suppose $r \leq s$.
The proof is by induction on $r$, with the result being true for $r=1$ by Proposition~\ref{Pf=B} 
when $f=2$ and Lemma~\ref{Lbasecase} when $f=0$.
Suppose the result is true for all trielliptic signatures $(r_1,s_1)$ with $1 \leq r_1 < r$.

Let $(r_2,s_2)$ be either (i) $(1,2)$ or (ii) $(1,1)$, with choice (i) mandated if $s=2r+1$ and choice (ii) 
mandated if $s=r$. Let $g_2=r_2+s_2$.
Let $r_1=r-r_2$ and $s_1=s-s_2$ and $g_1=r_1+s_1$.  
Note that $(r_1,s_1)$ is a trielliptic signature for $g_1$ and $r_1 \leq s_1$.

By the hypothesis, $0 \leq f \leq 2r$ is even.
Let $f_2$ be either (a) 2 or (b) 0, with choice (a) mandated if $f=2r$ and choice (b) mandated if $f=0$.
Let $f_1=f-f_2$.
Then $0 \leq f_1 \leq 2r_1$ and $f_1$ is even.

It follows that $\calt_{(r_i,s_i)}^{f_i}$ is non-empty and contains an irreducible component ${\mathcal S}_i$ with 
$\dim({\mathcal S}_i) = s_i-1 +f_i/2$ (by the inductive hypothesis when $i=1$, 
Propositions~\ref{Pf=B} and \ref{Cf=B-2} when $i=2$).
One can add a labeling of the smooth ramification points
by choosing an irreducible component $\tilde{{\mathcal S}}_i$ of $\tilde{\calt}_{(r_i,s_i)}$ above ${\mathcal S}_i$.

By construction,
$K=\kappa_{g_1, g_2}(\tilde{{\mathcal S}}_1 \times \tilde{{\mathcal S}}_2)$ is contained in 
$\overline{\calt}_{(r,s)}^f$ and $${\rm dim}(K)={\rm dim}({\mathcal S}_1) + {\rm dim}({\mathcal S}_2) = s - 2 +f/2.$$
Then $K$ is contained in a component $W$ of $\overline{\calt}_{(r,s)}^f$.  
By Proposition~\ref{Tpurity3}, ${\rm dim}(W) \geq s-1 + f/2$.
By Proposition~\ref{prop:upperbound}, ${\rm dim}(W) \leq s-1 + f/2$.
Thus ${\rm dim}(W) = s-1 + f/2$.

Finally, the generic point of $W$ is not contained in $K$.  
Since the generic points of ${\mathcal S}_1$ and ${\mathcal S}_2$ represent smooth curves
(this requires the hypothesis $p \not = 2$ for case (ii)), 
the generic point of $W$ is not contained in any other boundary component.
Thus the generic point of $W$ represents a smooth curve.
It follows that ${\mathcal S}=W \cap \calt_{(r,s)}^f$ is open and dense in $W$ and thus is 
non-empty with dimension $s-1 + f/2$.
\end{proof}

\begin{remark} \label{Rwhynototherp}
When $p \equiv 1 \bmod 3$, 
we were not able to prove an analogue of Theorem~\ref{Trsexistence}.
Then main reason is that if $f< g$ and 
if ${\mathcal S}$ is an irreducible component of $\bar \calt_{(r,s)}^f$, 
then Proposition~\ref{Tpurity3} states that  
${\rm dim}({\mathcal S}) \geq f$.  When $f=0$, the expected dimension is $0$,  
which makes it difficult to work with the $p$-rank $0$ stratum inductively.
\end{remark}

\section{Cases where all $p$-rank $0$ strata have the same dimension }

Suppose $p \equiv 2 \bmod 3$ is odd.  Let $g \geq 2$.  Let $(r,s)$ be a trielliptic signature for $g$.
Let $f$ be an even integer such that $0 \leq f \leq 2{\rm min}\{r,s\}$.
If ${\mathcal S}$ is an irreducible component of $\calt_{(r,s)}^f$, then $\dim({\mathcal S}) \geq \max\{r,s\}-1 +f/2$
by Proposition~\ref{Tpurity3}.

We proved in Theorem~\ref{Trsexistence} that 
$\calt_{(r,s)}^f$ is non-empty and contains an irreducible component
with $\dim({\mathcal S}) = \max\{r,s\}-1 +f/2$.
Motivated by a result in the hyperelliptic case \cite[Proposition 2]{GP},
we tried to prove that \emph{every} component of $\calt_{(r,s)}^f$ has the same dimension. 
This is true when ${\rm min}\{r, s\} =1$ by Lemma~\ref{Lbasecase}.
In this section, we prove it is also true when ${\rm min}\{r, s\} = 2$, see Corollary~\ref{Cr2total}.

Here are some of the reasons the trielliptic case is more difficult than the hyperelliptic case.
First, it is possible that there are components of $\bar{\calt}_{r,s}^f$ that are fully contained in the boundary.
This does not happen in the hyperelliptic case by \cite[Lemma~3.2]{APhypmono}.
We describe this in Section~\ref{Sboundaryproblem}.

Second, in the hyperelliptic case, 
every irreducible component of the $p$-rank $0$ stratum $\bar{\calh}^0_g$ 
contains the moduli point of a 
tree of $g$ (supersingular) elliptic curves \cite[Theorem 3.11(c)]{APhypmono}.
The analogous result in the trielliptic case is weaker.
By Theorem~\ref{Cbreakupf=0}, every irreducible component of 
$\bar{\calt}_{(r,s)}^0$ contains the moduli point of a tree of 
${\rm max}\{r,s\}$ trielliptic curves, but ${\rm max}\{r,s\}$ is strictly less than $g$.

\subsection{Balanced degenerations}

In this section, we introduce balanced degenerations which are helpful for finding an upper bound for the dimension of irreducible components.
The reason for the balanced condition is that $B(r,s)$ is not additive in general.
When $p \equiv 2 \bmod 3$ and $r \leq s$, then $B(r,s)=B(r_1,s_1)+B(r_2,s_2)$ if and only if $r_1 \leq s_1$ and $r_2 \leq s_2$.

\begin{definition} \label{Ddegenerate}
Let $S$ be an irreducible component of $\bar \calt^f_{(r,s)}$ with $r \leq s$.
We say $S$ {\em degenerates} to $\Delta((r_1,s_1)^{f_1}, (r_2, s_2)^{f_2})$ if
$S$ intersects $\kappa(\tilde \calt^{f_1}_{(r_1,s_1)} \times \tilde \calt^{f_2}_{(r_2,s_2)})$.

We say the degeneration is {\em balanced} if $r_1 \leq s_1$ and $r_2 \leq s_2$.
\end{definition}

In Definition \ref{Ddegenerate}, we implicitly require 
that $(r_1,s_1)$ and $(r_2,s_2)$ are trielliptic signatures,
that $r_1+r_2=r$ and $s_1+s_2=s$,
that $f_i$ are even with $0 \leq f_i \leq 2r_i$, 
and that $f_1 + f_2 \leq f$. 

\begin{proposition} \label{P:balanced}
Suppose $S$ has a balanced degeneration to $\Delta((r_1,s_1)^{f_1}, (r_2, s_2)^{f_2})$. 
Suppose, for $i=1,2$, that
\begin{equation} \label{hypdim}
{\rm dim}(\tilde \calt^{f_i}_{(r_i,s_i)})={\rm dim}(\bar \calt^{f_i}_{(r_i,s_i)})=s_i-1+f_i/2.
\end{equation}
Then $\dim(S)= s-1 +f/2$ and 
$S$ contains $\kappa(S_1 \times S_2)$, where $S_i$ denotes
a component of $\tilde \calt^{f_i}_{(r_i,s_i)}$ for $i=1,2$.
\end{proposition}

\begin{proof} 
By Theorem \ref{Tpurity3}, $\dim(S) \geq s - 1 +f/2$.
By Proposition \ref{prop:upperbound}, 
\[\dim(S) \leq \dim(\tilde \calt^{f_1}_{(r_1,s_1)})+ \dim(\tilde \calt^{f_2}_{(r_2,s_2)}) +1.\]
Since the degeneration is balanced, $r_i \leq s_i$.
By hypothesis,
\[{\rm dim}(\tilde \calt^{f_i}_{(r_i,s_i)})={\rm dim}(\bar \calt^{f_i}_{(r_i,s_i)})=s_i-1+f_i/2.\]
So
\[\dim(S) \leq (s_1-1 +f_1/2) + (s_2-1+f_2/2)+1 \leq s-1 +f/2.\] 

Thus $\dim(S) = s-1+f/2$. 
Furthermore, $S$ contains $\kappa(S_1 \times S_2)$
in order for equality to hold in the dimension count.
\end{proof}




There is an analogous result for the other boundary components, which we will not need in this paper.
In this case, we say $S$ {\em degenerates} to $\Xi((r_1,s_1)^{f_1}, (r_2, s_2)^{f_2})$ if
$S$ intersects $\lambda(\bar \calt^{f_1}_{(r_1,s_1);1} \times \tilde \calt^{f_2}_{(r_2,s_2);1})$.
In this case, we allow $(r_1,s_1)=(0,0)$ to be a valid trielliptic signature,
and require that $r_1+r_2=r-1$ and $s_1+s_2=s-1$, and $f_1 + f_2 \leq f-2$. 

\begin{proposition}
Suppose $S$ has a balanced degeneration to $\Xi((r_1,s_1)^{f_1}, (r_2, s_2)^{f_2})$.
For $i=1,2$, suppose \eqref{hypdim} is true.
Then $\dim(S)= s-1 +f/2$ and
$S$ contains $\lambda(S_1 \times S_2)$, where $S_i$ denotes
a component of $\bar \calt^{f_i}_{(r_i,s_i);1}$ for $i=1,2$.
\end{proposition}

\begin{proof}
The proof is almost the same as for Proposition~\ref{P:balanced}.
For a $\Xi$-degeneration, recall that $s_1+s_2=s-1$, $r_1+r_2=r-1$, and $f_1+f_2 \leq f-2$.
Marking an orbit increases the dimension by one, 
so ${\rm dim}(\bar \calt^{f_1}_{(r_1,s_1);1})=s_i+f_i/2$.
Then
\[\dim(S) \leq \dim(\bar \calt^{f_1}_{(r_1,s_1);1})+ \dim(\bar \calt^{f_2}_{(r_2,s_2);1}) +1,\]
so
\[\dim(S) \leq (s_1 +f_1/2) + (s_2+f_2/2)+1 \leq s-1 +f/2.\] 
\end{proof}

\subsection{A partial generalization of Proposition~\ref{Cf=B-2}} \label{Sapplication2}

When $f=B(r,s)-2$, then the $p$-rank $f$ stratum has codimension $1$ in 
$\calt_{(r,s)}$, by Proposition~\ref{Cf=B-2}.  
When $p \equiv 2 \bmod 3$ is odd, by Theorem~\ref{Trsexistence}, $\calt_{(r,s)}^f$
is non-empty for each $0 \leq f \leq 2{\rm min}\{r,s\}$ with $f$ even.

When $p \equiv 2 \bmod 3$ is odd, we would like to 
generalize Proposition~\ref{Cf=B-2} by showing that every component of the
$p$-rank $f=B(r,s)-4$ stratum has codimension $2$ in $\calt_{(r,s)}$.  
One reason this is hard to show is because it is not known whether 
the $p$-rank strata are nested in each other; specifically, it is not known whether  
every component of the $f=B(r,s)-4$ stratum is contained in the closure of the 
$f=B(r,s)-2$ stratum.

In the next result, we are able to extend Proposition~\ref{Cf=B-2} in this desired way
but only under the strong restriction that ${\rm min}\{r,s\}=2$.  

\begin{corollary} \label{Cr2total}
Let $p \equiv 2 \bmod 3$ be odd.  Let $(r,s)$ be a trielliptic signature with ${\rm min}\{r,s\}=2$.
Let $f=0$.
If ${\mathcal S}$ is an irreducible component of $\bar\calt_{(r,s)}^0$, then
$\dim({\mathcal S}) = \max\{r,s\}-1$.
\end{corollary}

In the rest of the section, we prove Corollary~\ref{Cr2total}. 
By symmetry, it suffices to suppose $r=2$; then $s=2,3,4,5$ and we handle these cases separately.

Corollary~\ref{Cr2total} is an application of 
Theorem~\ref{Cbreakupf=0}, which we restate in the trielliptic context:
suppose ${\mathcal S}$ is an irreducible component of the $p$-rank $0$ stratum
$\bar \calt_{(r,s)}^0$ of $\bar \calt_{(r,s)}$;
let $\sigma = {\rm dim}({\mathcal S})$; then 
there exists $\eta \in {\mathcal S}$ such that the curve 
$Y_\eta$ of compact type represented by $\eta$ is reducible,
with at least $\sigma + 1$ components,
such that the $\ZZ/3 \ZZ$-action stabilizes and acts non-trivially on each component.

\subsubsection{The case $r=2$ and $s=3$} \label{Sr=2}
 
\begin{lemma}\label{L:23}
If $S$ is an irreducible component of $\bar \calt_{(2,3)}^0$, then $\dim(S)=2$. 
\end{lemma}

\begin{proof} 
When the signature is $(2,3)$, the inertia type is $\bar{\gamma}=(1,1,2,2,2,2,2)$.

By Theorem \ref{Tpurity3}, $\dim(S) \geq 2$.
Since $\calt_{(2,3)}$ is affine, $S$ intersects either $\Delta_1=\Delta_4$ or $\Delta_2=\Delta_3$. 
If $S$ intersects $\Delta_2$, then $S$ has a balanced degeneration to $\Delta((1,1)^0, (1,2)^0)$.
By Lemma~\ref{Lbasecase}, the hypothesis in \eqref{hypdim} is true 
and so $\dim(S)=2$ by Proposition \ref{P:balanced}.

We assume that $S$ does not intersect $\Delta_2$ and that $\dim(S) \geq 3$ and find a contradiction. 
By Theorem~\ref{Cbreakupf=0}, $S$ contains a point $\eta$ representing a curve $Y_\eta$ with 
at least $4$ components.  Since $Y_\eta$ has genus $5$, it has three components of genus $1$ 
and one component $Y_0$ of genus $2$ (which is possibly reducible). 

In the dual graph of $Y_\eta$, we replace the vertex representing $Y_0$ by 
two vertices connected by a marked edge. 
This is illustrated in Figure \ref{fig:curve}: the schematic represents the four components of the curve, 
with the branch points marked by their canonical generators of inertia (the admissible condition implies that 
the two canonical generators of inertia are inverses at each ordinary double point);
the schematic in Figure  \ref{fig:23tail} represents the dual graph of $Y$.  
\begin{figure}[h]\centering 
	\caption{Picture of the singular curve $Y_\eta$ }
\begin{tabular}{c}
\includegraphics[width=0.4\linewidth]{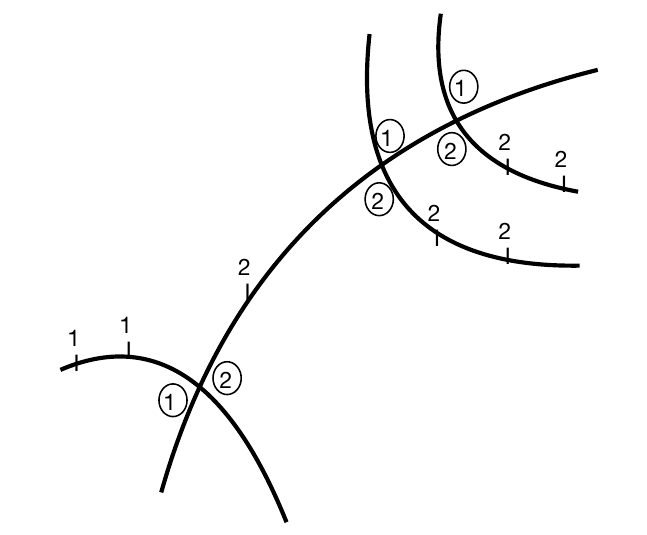} 
\end{tabular}\label{fig:curve}
\end{figure} 

\begin{figure}
	\caption{The dual graph of $Y_\eta$.}
\begin{tabular}{c}
   \xymatrix{
  {\times} \ar@{-}[dr] & {\bullet}\ar@{-}[dr] && {\bullet}\ar@{-}[dl]\\
  & {\bullet}\ar@{=}[r]  & {\times}&\\
  &111={\times}&222={\bullet}&} 
\end{tabular} \label{fig:23tail}
\end{figure}
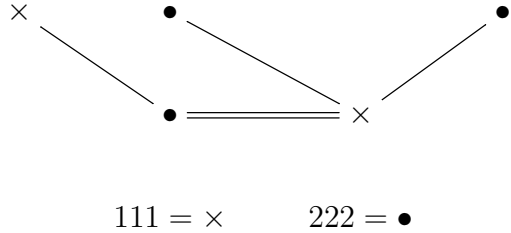

The moduli point of $Y$ is in $\kappa(\tilde \calt^0_{(1,0)} \times \tilde \calt^0_{(1,3)})$, 
but this is not a balanced degeneration.
Note that ${\rm dim}(\tilde \calt^0_{(1,0)})=0$ and ${\rm dim}(\tilde \calt^0_{(1,3)})=2$ by Lemma~\ref{Lbasecase}.

By Proposition \ref{prop:upperbound}, ${\rm dim}(S) \leq 3$.
If ${\rm dim}(S) = 2$, the proof is complete.
If ${\rm dim}(S) =3$, then there is a component $S_1$ of $\tilde \calt^0_{(1,0)}$ 
and a component $S_2$ of $\tilde \calt^0_{(1,3)}$
such that $S$ contains $\kappa(S_1 \times S_2)$.

Consider the PEL type Shimura variety ${\rm Sh}$ attached to the data of $\ell =3$ 
and the signature type $(1,3)$;
let $\Sigma$ be the irreducible component of ${\rm Sh}$ which contains the Torelli locus;
and let $\Sigma^0$ denote the $p$-rank $0$ stratum of $\Sigma$.
By Proposition~\ref{Pirreducible}, $\Sigma^0$ is irreducible.

Consider a (labeled) tree $\eta'$ of four elliptic curves (with dual graph in Figure \ref{fig:fourec});
this is a singular trielliptic curve with signature type $(1,3)$ and $p$-rank $0$.  
Its Jacobian is represented by a point of $\Sigma^0$, 
and thus by a point in the closure of $S_2$.  
\begin{figure}[h]
 \centering
\begin{tabular}{c}
	\xymatrix{
		{\bullet} \ar@{-}[dr] &  & {\bullet}\ar@{-}[dl] \\
		{\bullet}\ar@{-}[r]  & {\times}& &\\
		} 
\end{tabular} 
 \caption{A point $\eta'$ of $\bar \calt^0_{1,3}$}
 \label{fig:fourec}
 \end{figure}
This implies that the point $\kappa(S_1 \times \eta')$ is in $\Delta_2[S]$ which is a contradiction.
\end{proof}

\subsubsection{The case $r=2$ and $s=4$} \label{Sr=2}

\begin{lemma}\label{L:24}
If $S$ is an irreducible component of $\bar \calt_{(2,4)}^0$, then $\dim(S)=3$. 
\end{lemma}

\begin{proof}
By Theorem \ref{Tpurity3}, $\dim(S) \geq 3$.
Since $\calt_{(2,4)}$ is affine, 
$S$ intersects either $\Delta_1$ or $\Delta_2$ or $\Delta_3$. 
Then $S$ has a balanced degeneration to either $\Delta((0,1)^0, (2,3)^0)$ or $\Delta((1,1)^0, (1,3)^0)$ or $\Delta((1,2)^0, (1,2)^0)$.
The result follows from Proposition~\ref{P:balanced}, 
with the hypothesis in \eqref{hypdim} verified by Lemma~\ref{Lbasecase} and Lemma~\ref{L:23}.
\end{proof}

\subsubsection{The case $(r,s)=(2,5)$:} 

\begin{lemma}\label{L:25}
If $S$ is an irreducible component of $\bar \calt_{(2,5)}^0$, then $\dim(S)=4$. 
\end{lemma}

\begin{proof} 
By Theorem \ref{Tpurity3}, $\dim(S) \geq 4$.
We assume $\dim(S) \geq 5$ and find a contradiction. 
By Theorem~\ref{Cbreakupf=0}, 
there exists a point $\eta \in S$ such that the curve $Y=C_\eta$ has at least $6$ components.
Since $Y$ has genus $7$, it has five components of genus $1$ and 
one component of genus 2 (which is possibly reducible). 

In the dual graph of $Y$, we replace the vertex representing the curve of genus $2$ by 
two vertices connected by a marked edge. 
One possibility for the dual graph is seen in Figure \ref{fig:25}.

\begin{figure}[h]\centering \caption{}
\begin{tabular}{c}
\xymatrix{
 {\bullet} \ar@{-}[r] & {\times}\ar@{-}[d]\ar@{-}[r] &{\bullet}\ar@{=}[r]& {\times}\ar@{-}[d]\ar@{-}[r]& {\bullet} \\
 & {\bullet}  && {\bullet}&} \\ 

\end{tabular} \label{fig:25}
\end{figure}   

Regardless of the location of the marked edge, $\eta$ is in $\kappa(\tilde \calt^0_{(1,2)} \times \tilde \calt^0_{(1,3)})$.
Thus $S$ has a balanced degeneration to $\Delta((1,2)^0, (1,3)^0)$.
By Proposition \ref{P:balanced}(1), $\dim(S) \leq 4$, which gives a contradiction. 
Hence $\dim(S) = 4$.
\end{proof}

\subsubsection{The case $(r,s)=(2,2)$} \label{S:22}

The situation for the trielliptic signature $(2,2)$ is more complicated;
as seen in Example~\ref{Esingularproblem}, 
$\Delta_1[\bar \calt^0_{2,2}]$ has dimension $1$, which is larger than expected.
To avoid issues raised by problem, we use some results about the Shimura variety ${\rm Sh}$
attached to the data $\ell=3$ and signature type $(2,2)$.

\begin{lemma} \label{L:22}
If $S$ is an irreducible component of $\bar \calt^0_{(2,2)}$, then $\dim(S)=1$.
\end{lemma}

\begin{proof}
If $S$ intersects $\Delta_2$, then $S$ has a balanced degeneration to $\Delta((1,1)^0, (1, 1)^0)$.  
Then $\dim(S)=1$ by Proposition \ref{P:balanced}(1).

Suppose that $S$ does not intersect $\Delta_2$.
Then $S$ intersects $\Delta_1$ and $S$ degenerates to $\Delta((1,0)^0, (1, 2)^0)$.
Then $\dim(S)=1$ or $\dim(S)=2$ by Proposition \ref{P:balanced}(2).

Assume $\dim(S) =2$.  
Under this assumption, we will prove that $S$ intersects $\Delta_2$, 
which gives a contradiction.
By Theorem~\ref{Cbreakupf=0}, $S$ contains a point $\eta$ representing a trielliptic curve $C_\eta$ 
of compact type with at least $3$ components.  
Then $C_{\eta}$ has one irreducible component of genus 2
with two elliptic tails, 
since $S$ does not intersect $\Delta_2$.

Consider the Shimura variety ${\rm Sh}$
attached to the data $\ell=3$ and signature type $(2,2)$.
Let $\Sigma$ be the irreducible component of ${\rm Sh}$ containing the image of the Torelli locus.
Let $S'$ be the image of $S$ under the embedding $\bar \calt^0_{2,2} \to \Sigma$.
Then $S'$ is contained in the supersingular locus $N$ of $\Sigma$.
Each of the supersingular components of $\Sigma$ has dimension $2$ by \cite[Theorem~B]{howardpappas}.
It follows that $S'$ is a component of $N$.
By \cite[Theorem~3.12]{howardpappas}, every component of $N$ contains a point 
$\xi$ where $N$ intersects another component of $N$.
The point $\xi$ parametrizes an abelian variety $A$ which decomposes, together with the polarization,
into a direct sum of the form $A_1 \oplus A_2$ 
where ${\rm dim}(A_i)=2$ and $\ZZ[\zeta_3] \subset {\rm End}(A_i)$.
Then $A$ is the Jacobian of a curve in $\Delta_2[\bar \calt^0_{2,2}]$, 
which gives a contradiction.
\end{proof}

\section{Concluding remarks}

In Section~\ref{Scartier}, we complete the proof of the base case 
used in some results in the paper, such as Proposition~\ref{Cf=B-2}.
In the remaining sections, we explain some of the reasons why 
Theorem~\ref{Trsexistence} is hard to generalize.

\subsection{A computational approach to the base case} \label{Scartier}

In this section, we prove that $\calt^0_{(1,1)}$ is non-empty 
when $p \equiv 2 \bmod 3$ is odd;
this completes the proof of Lemma~\ref{Lbasecase}(1).

Suppose $\psi:Y \to {\mathbb P}^1$ is a $\ZZ/\ell \ZZ$-cover.
In \cite[Sections 2-3]{Elkin}, Elkin determines formulas for the action of the 
Cartier operator $\C$ on $H^0(Y, \Omega^1)$.
The Cartier operator is a semi-linear operator and the $p$-rank is the rank of its $g$th iterate.
For small $\ell$, $g$, and $p$, it is possible to compute the $p$-rank from Elkin's work.
We note that there are several choices involved in setting up the computation of iterates of the Cartier operator;
these have led to mistakes in the literature and the reader is advised to consult 
\cite{achterhowe} to avoid repeating these.

We restrict to the case that $\ell=3$.  Without loss of generality, 
we can suppose that $\psi$ is not branched above $\infty$.  
As in \eqref{Etrielliptic}, there is an equation $y^3 =  p_1(x)p_2(x)^2$ for $\psi$, where
$p_1(x), p_2(x) \in k[x]$ are squarefree, monic, and relatively prime. 
Let $d_1={\rm deg}(p_1(x))$ and $d_2={\rm deg}(p_2(x))$.
From Lemma~\ref{L:bijection}, recall that the trielliptic signature for $\psi$ is given by 
$(r,s)$ where $r=(2d_1+d_2-3)/3$ and $s = (d_1+2d_2-3)/3$.
Also $g=d_1+d_2 - 2$.

We restrict to the case $(r,s)=(1,1)$ and $g=2$.
Then $H^0(Y, \Omega^1) = \LL_1 \oplus \LL_2$ where 
\[\LL_2=\langle w_{1,1}:=dx/y \rangle, \ \LL_1 = \langle w_{2,1}:=p_2(x) dx/y^2 \rangle.\]
We suppose that $p \equiv 2 \bmod 3$;  
then the Cartier operator $\C$ permutes $\LL_1$ and $\LL_2$.

Let $e_1=(2p-1)/3$ and $e_2 = (p-2)/3$.  Note that $e_1+e_2=p-1$.  Let
\[h_1(x)=p_1(x)^{e_1} p_2(x)^{e_2}, \ h_2(x)=p_1(x)^{e_2}p_2(x)^{e_1}.\]
When $r=s=1$, then ${\rm deg}(h_1(x)) = {\rm deg}(h_2(x)) = 2p-2$.
In this case, for $i=1,2$, 
there is only one monomial $x^e$ in $h_i(x)$ such that $e \equiv -1 \bmod p$;
let $A_i$ be the coefficient of $x^{p-1}$ in $h_i(x)$.
By a special case of \cite[Theorem 3.4]{Elkin},
\begin{equation} \label{Eelkin}
\C(\omega_{1,j}) =A_1 \omega_{2,1}, \ 
\C(\omega_{2,j}) =A_2 \omega_{1,1}. 
\end{equation}

\begin{lemma} \label{Lcalculateg2}
\begin{enumerate}
\item If $p \equiv 2 \bmod 3$ is odd, then $\calt^0_{(1,1)}$ is non-empty: there exists a smooth trielliptic curve $Y$ over $\overline{\FF}_p$ with genus $2$ and $p$-rank $0$.  
\item If $p=2$, then $\calt^0_{(1,1)}$ is empty: 
there does not exist a smooth trielliptic curve $Y$ over $\overline{\FF}_2$ with genus $2$ and $2$-rank $0$. 
\end{enumerate}
\end{lemma}

\begin{proof}
\begin{enumerate}
\item
Let $p_1(x)=x^2-1$ and $p_2(x)=x^2+1$.
We compute that 
\begin{eqnarray*}
h_1(x) & = & (x^2-1)^{e_1}(x^2+1)^{e_2} \\
& = & \left(\sum_{j=0}^{e_1} (-1)^{e_1 -j}\dbinom{e_1}{j} x^{2j} \right)
\left(\sum_{i=0}^{e_2} \dbinom{e_2}{i} x^{2i} \right).
\end{eqnarray*}
By \eqref{Eelkin}, $\C(\omega_{1,1})= A_1 \omega_{2,1}$
where $A_1$ is the coefficient of $x^{p-1}$ in $h_1(x)$.
A pair $(i,j)$ contributes to this coefficient exactly when $2i+2j=p-1$, 
or $j=(p-1)/2 -i$.
Thus
\[A_1= \sum \limits_{i=0}^{e_2} (-1)^{\frac{p+1}{6}+i}\dbinom{e_1}{\frac{p-1}{2}-i}\dbinom{e_2}{i} .\] 
The sum has an even number of terms. 
One can check that the $i$-th and $(e_2-i)$-th terms cancel, because 
$((p-1)/2 -i)+((p-1)/2 - (e_2-i))=e_1$ and because $i$ and $e_2-i$ have opposite parities.
Thus $A_1=0$.  

The same argument, with the roles of $e_1$ and $e_2$ switched, shows that $A_2=0$.
Thus $\C$ is the zero operator (implying that $Y$ is superspecial) and so $Y$ has $p$-rank $0$. 

\item Let $p=2$.  After an automorphism of ${\mathbb P}^1$, one can suppose that
$p_1(x)=x^2+x+1$ and $p_2(x)=(x-1)(x-a)$ for some $a \in \bar \FF_2$.
Then $h_1(x)=p_1(x)$ and $h_2(x)=p_2(x)$ 
so $A_1=1$ and $A_2=-(a+1)$.
The determinant of the matrix of $\C$ is $a+1$.
If $Y$ has $2$-rank $0$ then this determinant is zero, so $a \equiv 1 \bmod 2$.
Then $p_2(x)$ is not square-free so $Y$ is singular. 
\end{enumerate}
\end{proof}

Elkin provides similar formulas for any prime degree $\ell$ and inertia type $\bar{\gamma}$;
unfortunately, this does not provide a viable way to study the $p$-rank strata when $g > 2$ and $p$ is large for any $\ell$, because the equations are algebraically complicated.

\subsection{Possibility of components fully contained in the boundary} \label{Sboundaryproblem}

Here is an important fact about the $p$-rank stratification of the hyperelliptic locus.
When $p$ is odd, the generic geometric 
point of each irreducible component of $\bar{\calh}^f_g$ represents a \emph{smooth} hyperelliptic curve, 
by \cite[Lemma 3.1]{APhypmono}.

The analogue of this fact is not true in general for the trielliptic locus.
For example, when $p=2$, then every point of $\bar{\calt}^0_{(1,1)}$ represents a \emph{singular}
trielliptic curve by Lemma~\ref{Lcalculateg2}(2).
This is the key reason for the restriction that $p$ is odd in Proposition~\ref{Cf=B-2} and Theorem~\ref{Trsexistence}.

We give some other examples where this type of problem may occur.

\begin{example} \label{Esingularproblem}
Let $p \equiv 2 \bmod 3$ be odd.  Let $r=s=2$ and $f=0$.
Suppose ${\mathcal S}$ is an irreducible component of $\bar \calt_{(2,2)}^0$.
By Proposition~\ref{Tpurity3}, ${\rm dim}({\mathcal S}) \geq 1$.

Every component $\til{S}_1$ of $\til{\calt}_{(1,0)}^0$ has dimension $0$ and every
component $\til{S}_2$ of $\til{\calt}_{(1,2)}^0$ has dimension $1$. 
This means that $\Delta_1[\bar \calt_{(2,2)}^0]$ contains a component $\kappa_{1,3}(\til{S}_1 \times \til{S}_2)$ of dimension $1$.
In other words, there is a component ${\mathcal S}$ of $\bar \calt_{(2,2)}^0$
which has one of the following two problems: either
${\rm dim}({\mathcal S}) > 1$ (bigger than expected)
or the generic point of ${\mathcal S}$ is contained in the boundary of $\bar \calt_{(2,2)}^0$. 
\end{example}

The reason for the problem in Example~\ref{Esingularproblem} is that the pair $(1,0)$ and $(1,2)$ is not balanced, as in Definition~\ref{Ddegenerate}.
The inductive step for producing smooth trielliptic curves of larger genus 
and given $p$-rank works only 
when the pair of signatures is balanced.

More generally, the problem seen in Example~\ref{Esingularproblem} 
arises when $p \equiv 2 \bmod 3$ is odd if $r \leq s \leq 2r-1$ and $f \leq 2r-2$
and when $p=2$ if $3 \leq r \leq s \leq 2r -2$.

\begin{example}
Let $p \equiv 1 \bmod 3$.  Let $r=s=2$ and $f=1$.
Suppose ${\mathcal S}$ is an irreducible component of $\bar \calt_{(2,2)}^1$.
By Proposition~\ref{Tpurity3}, ${\rm dim}({\mathcal S}) \geq 1$.

Every component $\til{S}_1$ of $\til{\calt}_{(1,1)}^0$ has dimension $0$ and every
component $\til{S}_2$ of $\til{\calt}_{(1,1)}^1$ has dimension $1$. 
This means that $\Delta_1[\bar \calt_{(2,2)}^1]$ contains a component 
$\kappa_{2,2}(\til{S}_1 \times \til{S}_2)$ of dimension $1$.
In other words, there is a component ${\mathcal S}$ of $\bar \calt_{(2,2)}^1$
which has one of the following two problems: either
${\rm dim}({\mathcal S}) > 1$ (bigger than expected)
or the generic point of ${\mathcal S}$ is contained in the boundary of $\bar \calt_{(2,2)}^1$. 
\end{example}

\subsection{The wild case $p=3$}

We include the case of trielliptic covers when $p=3$.

\begin{proposition} \label{Pwild3}
Suppose $k=\overline{{\mathbb F}}_3$. 
Then there exists a $\ZZ/3 \ZZ$-cover of ${\mathbb P}^1_k$ having genus $g$ and $3$-rank $0$ if and only if $g \not \equiv 2 \bmod 3$.
For $g \geq 2$ and $1 \leq e \leq g/2$, 
there exists a $\ZZ/3 \ZZ$-cover of ${\mathbb P}^1_k$ having genus $g$ and $3$-rank $2e$.
\end{proposition}

\begin{proof}
Suppose $\phi:Y \to {\mathbb P}^1_k$ is a $\ZZ/3\ZZ$-cover.
Then $\phi$ is given by an Artin-Schreier equation of the form $y^3-y=f(x)$ for some $f(x) \in k(x)$.
Suppose ${\rm div}_\infty(f(x))=\sum_{j=0}^{e} d_j P_j$ is the pole divisor of $f(x)$.
By Artin-Schreier theory, one can suppose that $3 \nmid d_j$ for each $j$.  Let $e_j=d_j+1$.

The Riemann-Hurwitz formula and the Deuring-Shafarevich formula imply that 
the genus of $Y$ is $g=-2+\sum_{j=0}^{e} e_j$ and the $3$-rank of $Y$ is $2e$ \cite[Lemma 2.6]{PZ12:artschprank}.
Consider the $3$-rank $f$ strata ${\mathcal A}_{g}^f$ of the moduli space of genus $g$ Artin-Schreier covers of ${\mathbb P}^1_k$. 
By \cite[Theorem 1.1]{PZ12:artschprank}, the irreducible components of ${\mathcal A}_{g}^f$ are in bijection with partitions $\sum_{j=0}^{e} e_j$ of 
$g+2$ such that $e_j \not \equiv 1 \bmod 3$ and the dimension of the component is $g-1-\sum_{j=0}^{e} \lfloor \frac{d_j}{3} \rfloor$.

Thus, when $p=3$, there exists a $\ZZ/3\ZZ$-cover of ${\mathbb P}^1_k$ having genus $g$ and $3$-rank $0$ if and only if $g \not \equiv 2 \bmod 3$.
By an inductive argument, one can show:
for $g \geq 2$ and $1 \leq e \leq g/2$, there exists a partition $\sum_{j=0}^r e_j$ of $g+2$ into $r+1$ positive integers such that $e_j \not \equiv 1 \bmod 3$
and thus there exists a $\ZZ/3\ZZ$-cover of ${\mathbb P}^1_k$ having genus $g$ and $3$-rank $2r$.
\end{proof}


\bibliographystyle{plain}
\bibliography{POW316}
\end{document}